\RequirePackage{fix-cm}
\documentclass[12pt]{article}
\usepackage{graphicx}
\usepackage{amssymb,mathrsfs}
\usepackage{mathtools}
\usepackage{tikz}
\usetikzlibrary{matrix,calc}
\usetikzlibrary{decorations.markings,decorations.pathreplacing}
\usetikzlibrary{patterns}
\usepackage{subfigure}
\usepackage{cancel}
\usepackage{url}
\usepackage{calc}
\usepackage{booktabs}
\usepackage{authblk}
\usepackage{amsmath}
\usepackage{amsthm}
\tikzset{square matrix/.style={
    matrix of nodes,
    column sep=-\pgflinewidth, row sep=-\pgflinewidth,
    nodes={draw,
      minimum height=15pt,
      anchor=center,
      text width=13pt,
      align=center,
      inner sep=0pt
    },
  },
  square matrix/.default=2cm
}

\def\eref#1{$(\ref{#1})$}
\def\lref#1{Lemma~$\ref{#1}$}
\def\tref#1{Theorem~$\ref{#1}$}
\def\fref#1{Figure~$\ref{#1}$}
\def\cyref#1{Corollary~$\ref{#1}$}
\def\sref#1{Section~$\ref{#1}$}
\def\tabref#1{Table~$\ref{#1}$}

\def\Z{\mathbb{Z}}

\def\G{\mathscr{G}}
\def\cov{\mathscr{C}}
\def\eps{\varepsilon}

\def\Urcs{\mathcal{U}_\mathrm{RCS}}
\def\Urc{\mathcal{U}_\mathrm{RC}}
\def\Urs{\mathcal{U}_\mathrm{RS}}
\def\Ucs{\mathcal{U}_\mathrm{CS}}
\def\Ur{\mathcal{U}_\mathrm{R}}
\def\Uc{\mathcal{U}_\mathrm{C}}
\def\Us{\mathcal{U}_\mathrm{S}}
\def\Vr{\mathcal{V}_\mathrm{R}} 
\def\Vc{\mathcal{V}_\mathrm{C}} %
\def\Vs{\mathcal{V}_\mathrm{S}} %

\newcommand{\unkn}{\small\bullet}
\newcommand{\pmax}{p_{\text{\normalfont max}}}
\newcommand{\qmin}{q_{\text{\normalfont min}}}

\renewcommand{\geq}{\geqslant}
\renewcommand{\leq}{\leqslant}
\renewcommand{\ge}{\geqslant}
\renewcommand{\le}{\leqslant}

\newtheorem{theo}{Theorem}
\newtheorem{lemm}{Lemma}

\newtheorem{corol}{Corollary}
\newtheorem{conj}{Conjecture}

\usepackage{scrextend}

\title{Covers and partial transversals of Latin squares\thanks{The authors thank Adel Kazemi for bringing the idea of studying covers to their attention.  Best was supported by Endeavour Postgraduate Scholarship and the NSERC CGS-D. Stones was supported by her NSFC Research Fellowship for International Young Scientists (grant number: 11550110491) and the Thousand Youth Talents Plan in Tianjin. Wanless' research was supported by ARC grant DP150100506. Corresponding authors' email addresses: \texttt{trent.marbach@gmail.com} (T.\ Marbach) and \texttt{rebecca.stones82@gmail.com} (R.\ J.\ Stones)}
}

\author[1]{Darcy~Best}
\author[1]{Trent~Marbach}
\author[2]{Rebecca~J.~Stones}
\author[1]{Ian~M.~Wanless}

\affil[1]{\small School of Mathematical Sciences, Monash University, Australia
}
\affil[2]{\small Nankai-Baidu Joint Laboratory, College of Computer and Control Engineering \& College of Software, Nankai University, China
}

\begin{document}

\maketitle

\begin{abstract}
We define a {\em cover} of a Latin square to be a set of entries that
includes at least one representative of each row, column and symbol. A
cover is {\em minimal} if it does not contain any smaller cover.  A
partial transversal is a set of entries that includes at most one
representative of each row, column and symbol. A partial transversal
is {\em maximal} if it is not contained in any larger partial
transversal.
We explore the relationship between covers and partial transversals.

We prove the following:
(1)~The minimum size of a cover in a Latin square of order $n$
is $n+a$ if and only if the maximum size of a
partial transversal is either $n-2a$ or $n-2a+1$.  
(2)~A minimal cover in a Latin square of order $n$ has size at most 
$\mu_n=3(n+1/2-\sqrt{n+1/4})$.  
(3)~There are infinitely many orders $n$ for which there exists a 
Latin square having a minimal cover of every size from $n$ to $\mu_n$. 
(4)~Every Latin square of order $n$ has a minimal cover of a size
which is asymptotically equal to $\mu_n$.
(5)~If $1\le k\le n/2$ and $n\ge5$ then there is a Latin square of order 
$n$ with a maximal partial transversal of size $n-k$.
(6)~For any $\eps>0$, asymptotically almost all Latin squares have
no maximal partial transversal of size less than $n-n^{2/3+\eps}$.
\end{abstract}

\section{Introduction}


A \emph{Latin square} of order $n$ is an $n \times n$ matrix
containing $n$ symbols such that each row and each column contains one copy of each symbol. Unless otherwise specified, we use $\Z_n$ as the symbol set and also use $\mathbb{Z}_n$ to index the rows and columns. Where convenient (such as when embedding a Latin square inside a larger one), we consider $\Z_n$ to be the set of integers $\{0,\dots,n-1\}$ rather than a set of congruence classes. For a Latin square $L=[L_{ij}]$, we define $E(L)=\{(i,j,L_{ij}) : i,j \in \mathbb{Z}_n\}$ to be the set of \emph{entries}. The set of all entries in a row, all entries in a column or all entries containing a given symbol is called a \emph{line}. In particular, a Latin square of order $n$ contains exactly $3n$ lines. We say a line $\ell$ is \emph{represented} by an entry $\mathbf{e}$ whenever $\mathbf{e} \in \ell$.  For a set of entries $\cov \subseteq E(L)$, we say $\ell$ is \emph{represented} by $\cov$ whenever $|\cov \cap \ell| \geq 1$, and we say it is represented $|\cov \cap \ell|$ times by $\cov$.  If $\cov \cap \ell=\{\mathbf{e}\}$, we say that $\ell$ is \emph{uniquely represented} by $\mathbf{e}$.  We define a $c$-\emph{cover} as a $c$-subset of $E(L)$ in which every line is represented.  In order for a Latin square of order $n$ to have a $c$-cover, we must have $c \geq n$.

A \emph{partial transversal} of deficit $d$ is an $(n-d)$-subset of $E(L)$ in which every line is represented at most once.  Since an entry $(r,c,s)$ in a partial transversal uniquely represents three lines (its row, column and symbol), a partial transversal of deficit $d$ represents exactly $3(n-d)$ lines.  A \emph{transversal} is a partial transversal of deficit $0$.  \fref{fi:examples} gives examples of a partial transversal, a transversal and a cover.

In a Latin square $L$, we say a cover $\cov$ of $L$ is \emph{minimal} if, for all $\mathbf{e} \in \cov$, the set $\cov \setminus \{\mathbf{e}\}$ is not a cover.  If $\cov$ is not minimal, then it has a \emph{redundant} entry $\mathbf{e} \in \cov$ for which $\cov \setminus \{\mathbf{e}\}$ is also a cover.  We also say $\cov$ is \emph{minimum} if every cover of $L$ has size at least $|\cov|$.  We say a partial transversal $T$ of $L$ is \emph{maximal} if, for all $\mathbf{e} \in E(L) \setminus T$, the set $T \cup \{\mathbf{e}\}$ is not a partial transversal. We stress that the maximality of a partial transversal $T$ is always relative to the whole Latin square $L$, even when we locate $T$ inside some proper subset of $E(L)$.

\begin{figure}[htp]
\centering
\begin{tikzpicture}
\matrix[square matrix]{
|[fill=blue!50]| 0 & 1 & 2 & 3 \\
3 & |[fill=blue!50]| 2 & 1 & 0 \\
2 & 3 & 0 & |[fill=blue!50]| 1 \\
1 & 0 & 3 & 2 \\
};
\end{tikzpicture}
\quad
\begin{tikzpicture}
\matrix[square matrix]{
|[fill=blue!50]| 0 & 1 & 2 & 3 \\
3 & 2 & |[fill=blue!50]| 1 & 0 \\
2 & |[fill=blue!50]| 3 & 0 & 1 \\
1 & 0 & 3 & |[fill=blue!50]| 2 \\
};
\end{tikzpicture}
\quad
\begin{tikzpicture}
\matrix[square matrix]{
|[fill=blue!50]| 0 & 1 & 2 & |[fill=blue!50]| 3 \\
3 & |[fill=blue!50]| 2 & 1 & 0 \\
2 & 3 & 0 & |[fill=blue!50]| 1 \\
1 & 0 & |[fill=blue!50]| 3 & 2 \\
};
\end{tikzpicture}
\caption{\label{fi:examples}A Latin square of order $n=4$ where we highlight a partial transversal of deficit $1$ (left), a transversal (middle), and an $(n+1)$-cover (right).}
\end{figure}

There are many tantalising open questions regarding transversals \cite{transurv}. One of the more famous problems is Brualdi's Conjecture, which asserts that every Latin square possesses a \emph{near transversal}, that is, a partial transversal of deficit $1$. The current best result in this direction is due to Shor and Hatami \cite{HS08} who showed that every Latin square has a partial transversal of deficit $O(\log^2n)$. There are a great many Latin squares that do not possess transversals \cite{CW17}, although no such example of odd order is known.  In fact, Ryser \cite{Ryser} conjectured that there is no transversal-free Latin square of odd order.  In this paper, we introduce the notion of {\em covers} with the primary aim of using them to facilitate the study of transversals.


Pippenger and Spencer \cite{PS89} showed a very powerful and general result that includes covers of Latin squares as a special case. They showed that as $n\rightarrow\infty$, the entries of a Latin square of order $n$ can be decomposed into $n-o(n)$ covers. In particular, this means that all Latin squares have a cover of size $n+o(n)$. A better upper bound on the size of the smallest cover is given in \cyref{cy:allLSsmallcov}.

A Latin square $L$ of order $n$ is equivalent to a tripartite $3$-uniform hypergraph with $n$ vertices in each part (corresponding respectively to rows, columns and symbols) and $n^2$ hyperedges (corresponding to the entries of $L$). In this framework, a cover of $L$ is precisely an edge cover (a set of hyperedges whose union covers the vertex set) of this hypergraph.  Alternatively, $L$ can be considered as an $n$-uniform hypergraph of order $n^2$ with edges that are precisely the $3n$ lines of $L$; a cover of $L$ is precisely a vertex cover (a set of vertices that intersects every edge) of this hypergraph.  This relationship with hypergraph covers is one justification for our choice of terminology.

Another reason for our terminology is the connection to covering codes. Writing each entry of $L$ as a codeword of length 3 over an alphabet of size $n$, we obtain a maximal distance separable code (the connection between Latin squares and MDS codes was described by McWilliams and Sloane \cite{McS} as ``one of the most fascinating chapters in all of coding theory'').  By extending the alphabet to $Q=\mathbb{Z}_n \cup \{\infty\}$, a cover $\cov \subseteq E(L)$ of a Latin square $L$ of order $n$ is, in the sense of \cite{AF03}, a special case of a $2$-cover in $Q^3$  of the set $\{(i,\infty,\infty) : i \in \mathbb{Z}_n\} \cup \{(\infty,j,\infty) : j \in \mathbb{Z}_n\} \cup \{(\infty,\infty,k) : k \in \mathbb{Z}_n\}$.  It has the additional property that $\cov$ is contained in $E(L)$ for some Latin square $L$, which implies codewords in $\cov$ do not contain $\infty$.  


A Latin square $L$ of order $n$ also has a natural representation as an $n^2$-vertex graph, called a \emph{Latin square graph} which we denote $\Gamma_L$, with vertex set $E(L)$ and an edge between two distinct entries whenever they share a row, column or symbol.  An example is shown in \fref{fi:LRgraph}.  The graph $\Gamma_L$ is thus the union of $3n$ cliques of size $n$ (one for each line), and a cover in $L$ is equivalent to a selection of vertices in $\Gamma_L$ in which each of these cliques has at least one representative.  A cover of $L$ does not necessarily map to a vertex cover of $\Gamma_L$ (the example in \fref{fi:LRgraph} is not a vertex cover of $\Gamma_L$).

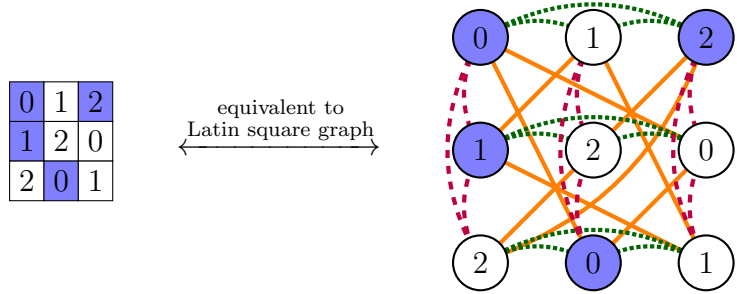
\begin{figure}[htp]
\centering
$\begin{array}{c}
\begin{tikzpicture}
\matrix[square matrix]{
|[fill=blue!50]| 0 & |[fill=white]| 1 & |[fill=blue!50]| 2 \\
|[fill=blue!50]| 1 & |[fill=white]| 2 & |[fill=white]| 0 \\
|[fill=white]| 2 & |[fill=blue!50]| 0 & |[fill=white]| 1 \\
};
\end{tikzpicture}
\end{array}
\quad
\xleftrightarrow{\substack{\text{equivalent to} \\ \text{Latin square graph}}}
\quad
\begin{array}{c}
\begin{tikzpicture}[scale=1.5]

\coordinate (02) at (0,2);
\coordinate (01) at (0,1);
\coordinate (00) at (0,0);
\coordinate (12) at (1,2);
\coordinate (11) at (1,1);
\coordinate (10) at (1,0);
\coordinate (22) at (2,2);
\coordinate (21) at (2,1);
\coordinate (20) at (2,0);

\draw[ultra thick, orange] (20) to (12) to (01);
\draw[ultra thick, orange] (20) to (01);
\draw[ultra thick, orange] (21) to (10) to (02);
\draw[ultra thick, orange] (21) to (02);
\draw[ultra thick, orange] (22) to (11) to (00);
\draw[ultra thick, orange] (22) to[bend left=25] (00);

\draw[ultra thick, green!40!black, densely dotted] (02) to[bend left=30] (12) to[bend left=30] (22);
\draw[ultra thick, green!40!black, densely dotted] (02) to[bend left=30] (22);
\draw[ultra thick, green!40!black, densely dotted] (01) to[bend left=30] (11) to[bend left=30] (21);
\draw[ultra thick, green!40!black, densely dotted] (01) to[bend left=30] (21);
\draw[ultra thick, green!40!black, densely dotted] (00) to[bend left=30] (10) to[bend left=30] (20);
\draw[ultra thick, green!40!black, densely dotted] (00) to[bend left=30] (20);

\draw[ultra thick, purple, dashed] (20) to[bend left=30] (21) to[bend left=30] (22);
\draw[ultra thick, purple, dashed] (20) to[bend left=30] (22);
\draw[ultra thick, purple, dashed] (10) to[bend left=30] (11) to[bend left=30] (12);
\draw[ultra thick, purple, dashed] (10) to[bend left=30] (12);
\draw[ultra thick, purple, dashed] (00) to[bend left=30] (01) to[bend left=30] (02);
\draw[ultra thick, purple, dashed] (00) to[bend left=30] (02);

\node[draw,thick,circle,fill=blue!50] at (0,2) (02) {$0$};
\node[draw,thick,circle,fill=white] at (1,2) (22) {$1$};
\node[draw,thick,circle,fill=blue!50] at (2,2) (22) {$2$};

\node[draw,thick,circle,fill=blue!50] at (0,1) (01) {$1$};
\node[draw,thick,circle,fill=white] at (1,1) (11) {$2$};
\node[draw,thick,circle,fill=white] at (2,1) (21) {$0$};

\node[draw,thick,circle,fill=white] at (0,0) (00) {$2$};
\node[draw,thick,circle,fill=blue!50] at (1,0) (10) {$0$};
\node[draw,thick,circle,fill=white] at (2,0) (20) {$1$};
\end{tikzpicture}
\end{array}
$
\caption{\label{fi:LRgraph}Converting between a Latin square $L$ and the equivalent Latin square graph $\Gamma_L$, with an $(n+1)$-cover highlighted in both.  Edge colours are added to indicate the relationship between neighbouring entries (dotted for the same row, dashed for the same column and solid for the same symbol).}
\end{figure}








Any cover of $L$ maps to a dominating set of $\Gamma_L$.  In fact, any cover of $L$ corresponds to a $3$-dominating set of $\Gamma_L$, i.e., any entry outside the $3$-dominating set has $3$ or more neighbours inside the $3$-dominating set \cite[Sec.\,7.1]{HHS5} (see also \cite{kdomsurvey}).  The converse is not true, i.e., not every $3$-dominating set is a cover: a $3$-dominating set (actually a $4$-dominating set) is formed in $\Gamma_L$ by the entries with symbols $0$ and $1$ in any Latin square $L$ of order $n \geq 2$. Yet, when $n \geq 3$, this $4$-dominating set does not cover the symbol $2$.  A cover therefore corresponds to a special kind of $3$-dominating set, where each $n$-clique (arising from each line in the Latin square) has a representative in the cover.

Let $L$ be a Latin square of order $n \geq 3$.  The domination number of $\Gamma_L$, i.e., the size of the smallest dominating set of $\Gamma_L$, denoted $\gamma(\Gamma_L)$, is less than $n$: to form an $(n-1)$-entry dominating set, select all but one of the entries with symbol $0$.  In fact, $\gamma(\Gamma_L)$ will likely be smaller than $n-1$, since any maximal partial transversal corresponds to a dominating set in $\Gamma_L$.  However, for a $3$-dominating set of cardinality $a$ to exist in $\Gamma_L$, we must have \[a\, 3(n-1) \geq 3(n^2-a)\] since each of the $a$ entries in the $3$-dominating set dominates at most $3(n-1)$ vertices, and there are $n^2-a$ entries dominated at least $3$ times each.  This implies that $a \geq n$, implying the $3$-domination number of $\Gamma_L$, denoted $\gamma_3(\Gamma_L)$, is strictly greater than the domination number, i.e., $\gamma_3(\Gamma_L)>\gamma(\Gamma_L)$.  (In fact, $\gamma_3(G)>\gamma(G)$ holds for all graphs $G$ with minimum degree at least $3$ \cite[Cor.~7.2]{HHS5}.)


For each Latin square $L$ there are six \emph{conjugate} squares obtained by uniformly permuting the three coordinates in $E(L)$. An \emph{isotopism} of $L$ is a permutation of its rows, permutation of its columns and permutation of its symbols. The resulting square is said to be {\em isotopic\/} to $L$.  The {\em isotopism class} of $L$ is the set of Latin squares isotopic to $L$. The {\em autotopism group} of $L$ is the group of isotopisms that map $L$ to itself. The {\em species} of $L$ is the set of squares that are isotopic to some conjugate of $L$. 

A theme in our work is to explore a loose kind of duality between covers and partial transversals. In \sref{s:duality} we demonstrate some relationships between the sizes of maximum partial transversals and minimum covers, and between the numbers of these objects. In \sref{s:maxmincov} we look at the other end of the spectrum, namely small maximal partial transversals and large minimal covers. Here we find less of a connection. We show that Latin squares of a given size have little variation in the size of their largest minimal covers, but can vary significantly in the size of their smallest maximal partial transversals. In \sref{s:conclude} we summarise our achievements and discuss possible directions for future research.

\section{Covers and partial transversals}\label{s:duality}

In this section, we explore some basic relationships between covers and partial transversals.  We first consider how to turn a partial transversal into a cover. Throughout, we will use $\unkn$ in an entry when its value is irrelevant to our argument. For example, $(i,j,\unkn)$ is the entry in row $i$ and column $j$, while $(\unkn,\unkn,k)$ is an arbitrary entry with symbol $k$.

\begin{theo}\label{th:PTtoCover}
In a Latin square $L$ of order $n \geq 2$, any partial transversal $T$ of deficit $d$ is contained in an $(n+\lceil d/2 \rceil)$-cover.  Moreover, if $T$ is maximal, then the smallest cover containing $T$ has size $n+\lceil d/2 \rceil$.
\end{theo}

\begin{proof}
We begin assuming $T$ is maximal, in which case any entry in $E(L) \setminus T$ covers at most two previously uncovered lines.
Let $r_1,\dots,r_d$, $c_1,\dots,c_d$ and $s_1,\dots,s_d$ denote, respectively, the rows, columns and symbols that are unrepresented in $T$. Start by setting $\cov = T$. Then for $i\in\{1,\dots,\lfloor d/2\rfloor\}$ we add
$(r_{2i-1},c_{2i-1},\unkn)$, $(r_{2i},\unkn,s_{2i-1})$ and $(\unkn,c_{2i},s_{2i})$ to $\cov$. Finally, if $d$ is odd we add $(r_d,c_d,\unkn)$ and $(\unkn,\unkn,s_d)$ to $\cov$. This produces a cover of size $n-d+\lceil 3d/2\rceil=n+\lceil d/2\rceil$.
As we covered the maximum possible number of uncovered lines at each step, no smaller cover contains~$T$.

If $T$ is not maximal, then the above approach gives a cover $\cov$ of size at most $n+\lceil d/2 \rceil$, since there may be duplication among the entries that are added. Assuming $n\ge2$, we can simply add entries from $E(L) \setminus \cov$ until we have a cover of size $n+\lceil d/2 \rceil$.
\end{proof}

Since Shor and Hatami \cite{HS08} have shown the existence of a partial transversal with small deficit, we immediately get:

\begin{corol}\label{cy:allLSsmallcov}
Every Latin square of order $n$ has a cover of size $n+O(\log^2 n)$.
\end{corol}

We now consider how to turn a cover into a partial transversal.

\begin{theo}\label{th:CovertoPT}
Let $L$ be Latin square of order $n \geq 1$. Any $(n+a)$-cover of $L$ contains a partial transversal of deficit $2a$.
\end{theo}

\begin{proof}
Let $R$, $C$ and $S$ respectively be $n$-subsets of an $(n+a)$-cover $\cov$ in which each row, column and symbol is (necessarily uniquely) represented.  Note that $T = R \cap C \cap S$ is a partial transversal of $L$. Since $|\cov|=n+a$ and $|\cov \setminus R|=|\cov \setminus C|=|\cov \setminus S|=a$, $$T = R \cap C \cap S=\cov \setminus \big( (\cov \setminus R) \cup (\cov \setminus C) \cup (\cov \setminus S) \big)$$ has size at least $n-2a$, so has deficit at most $2a$. Finally, if $T$ has a smaller deficit, we can delete entries to obtain deficit exactly $2a$.
\end{proof}

For any $(n+1)$-cover of a Latin square $L$, the corresponding $n+1$ vertices of the Latin square graph $\Gamma_L$ induce a subgraph with $3$ edges. (This is an example of a partial Latin square graph \cite{FalconStones}. A {\em partial Latin square} is a matrix in which entries are either empty or contain a single symbol, and no symbol is repeated within any row or column. Alternatively, a partial Latin square can be viewed as a set of triples where no two triples agree in more than one coordinate.)  Ignoring isolated vertices and edge colours, there are only $5$ such graphs, which we denote $G_1,\ldots,G_5$, depicted in \fref{fig:subgraphs}. We will refer to these graphs as being the graph \emph{induced} by the cover (specifically, this terminology ignores isolated vertices). Taking a conjugate of $L$ permutes the edge colours in the graph induced by a cover, which does not change the type of graph according to our classification.

\begin{figure}[htb]
\centering
\includegraphics{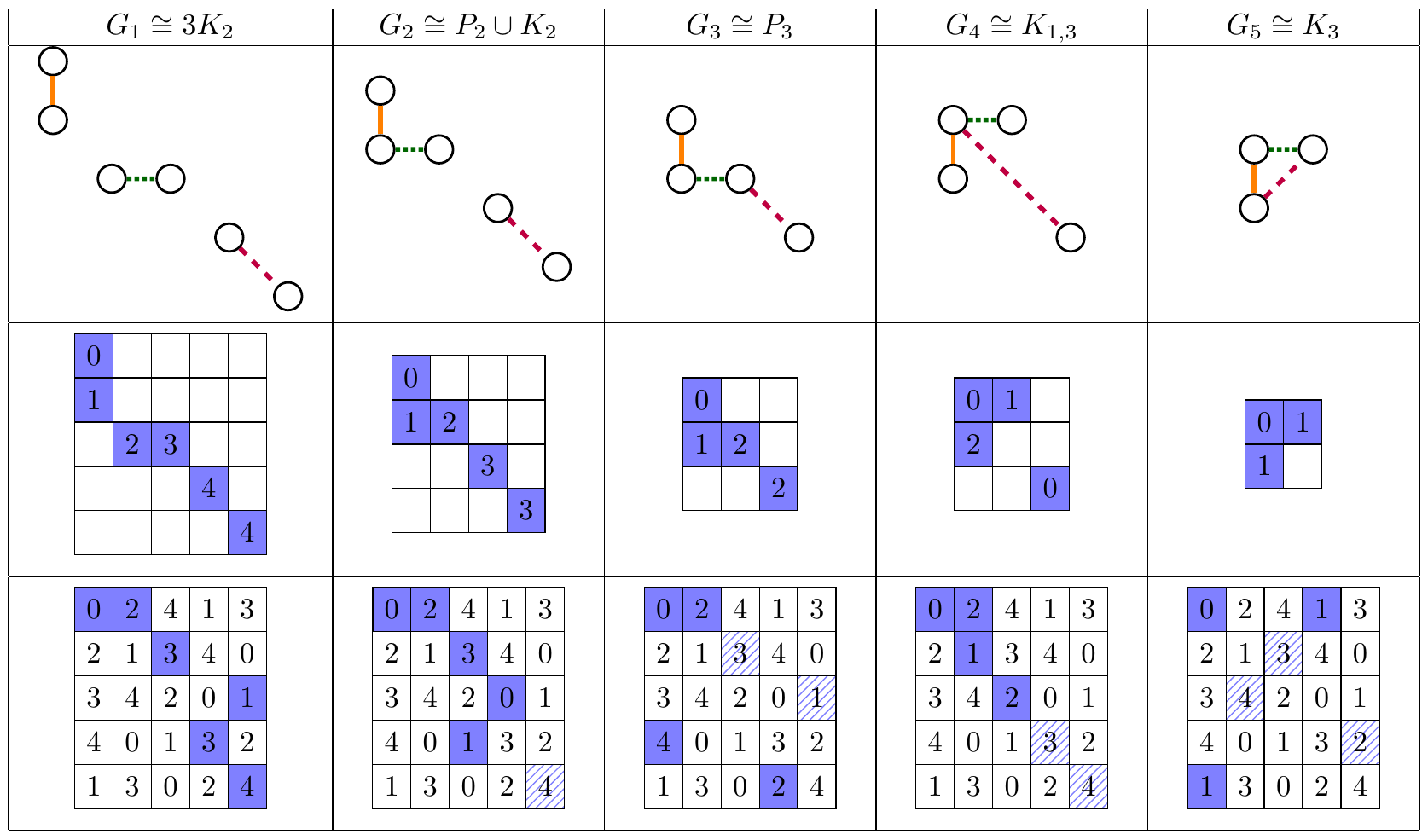}
\caption{\label{fig:subgraphs}\textit{Top row}: The five possible non-isomorphic subgraphs induced by an $(n+1)$-cover of a Latin square graph.
\textit{Middle row}: Depicting how the subgraphs $G_1,\ldots,G_5$ can arise in a cover.
\textit{Bottom row}: An example of a Latin square that simultaneously contains different covers that induce the five graph structures, $G_1,\ldots,G_5$.}
\end{figure}

A consequence of \tref{th:PTtoCover} is that any Latin square of order $n \geq 2$ with a partial transversal of deficit $1$ has an $(n+1)$-cover.  Thus, if Brualdi's Conjecture is true, then all Latin squares of order $n \geq 2$ have an $(n+1)$-cover.  A converse of this statement is not immediate since we cannot always delete $2$ entries from an $(n+1)$-cover to give a partial transversal of deficit $1$; see \fref{fig:subgraphs} (under graph $G_1$) for an example.  However, Theorems~\ref{th:PTtoCover} and~\ref{th:CovertoPT} imply that a Latin square $L$ of order $n \geq 2$ has a partial transversal of deficit $2$ if and only if it has an $(n+1)$-cover.  We now extend this observation to minimum covers.

\begin{theo}\label{th:minCoverPT}
Let $L$ be a Latin square of order $n \geq 2$.  The minimum size of a cover of $L$ is $n+a$ if and only if the minimum deficit of a partial transversal of $L$ is either $2a$ or $2a-1$.
\end{theo}

\begin{proof}
First suppose that $L$ has an $(n+a)$-cover and no smaller cover. By 
\tref{th:CovertoPT}, there is a partial transversal of deficit $2a$.  
If $L$ has a partial transversal of deficit at most $2a-2$, then
\tref{th:PTtoCover} implies there is a cover of size at most $n+a-1$,
which we are assuming is not the case. Hence, the
minimum deficit of a partial transversal is either $2a$ or $2a-1$.

For the converse, suppose the minimum deficit of a partial
transversal is either $2a$ or $2a-1$. By \tref{th:PTtoCover}, there 
is an $(n+a)$-cover. If there is a cover of size at most $n+a-1$,
then \tref{th:CovertoPT} implies there is a partial transversal of
deficit at most $2a-2$, which we are assuming is not the case.
\end{proof}

For the $a=0$ case in \tref{th:minCoverPT}, a transversal of a Latin square of order $n$ is also an $n$-cover.  For the $a=1$ case, cyclic group tables of even order are examples for which the minimum size of a cover is $n+1$ and the minimum deficit of a partial transversal is $1$. Brualdi's Conjecture implies the minimum size of a cover of an order-$n$ Latin square is $n$ or $n+1$.

\fref{fig:subgraphs} also includes an example of a Latin square of order $5$ in which all five of the possible induced subgraphs are achieved by different $(n+1)$-covers.  We make the following observations about deleting vertices from the graphs in \fref{fig:subgraphs}.

\begin{itemize}
 \item For graph $G_4$, we can delete one vertex to create an edgeless graph, so deleting the corresponding entry from the $(n+1)$-cover gives a transversal.  Thus $(n+1)$-covers that induce $G_4$  are not minimal, unlike the other four graphs ($G_1$, $G_2$, $G_3$ and $G_5$).
 \item For graphs $G_2,\ldots,G_5$, we can delete two vertices to create an edgeless graph, and deleting the corresponding entries from the $(n+1)$-cover gives a near-transversal.
 \item For graph $G_1$, we must delete at least $3$ vertices to create an edgeless graph.
 \item For any vertex $v$ of any of the five graphs $G_1,\ldots,G_5$, it is possible to delete $3$ or fewer vertices to create an edgeless graph without deleting $v$.  Thus when $n \geq 2$, every entry in an $(n+1)$-cover belongs to a partial transversal of deficit $2$.
\end{itemize}

We define $q_i=q_i(L)$ to be the number of $(n+1)$-covers that induce $G_i$ in a Latin square $L$. Across all isotopism classes of order $n \leq 8$, we found all $(n+1)$-covers. \tabref{tab:avg-q_i} lists the average number of $(n+1)$-covers that induce each graph across these isotopism classes. \tabref{tab:avg-q_i} also shows the fewest number of $(n+1)$-covers found of each of the 5 types. It is interesting to note that for each $n$, the number of all $(n+1)$-covers is fairly consistent across the Latin squares of order $n$ (in the sense that the range is small compared to the average). This is not true, for example, for the number of transversals.

\begin{table*}[ht]\centering
\begin{tabular}{@{}lrrrrrr@{}}\toprule
& \multicolumn{6}{c}{Average number of covers} \\
\cmidrule{2-7}
& $G_1$ & $G_2$ & $G_3$ & $G_4$ & $G_5$ & All \\ \midrule
$n = 5$ & 62& 90& 54& 180& 14& 400\\
$n = 6$ & 165& 889& 526& 229& 60& 1\,871\\
$n = 7$ & 1\,137& 4\,615& 2\,413& 900& 132& 9\,199\\
$n = 8$ & 8\,067& 24\,675& 10\,163& 3\,419& 483& 46\,808\\
\bottomrule
\end{tabular}

\medskip

\begin{tabular}{@{}lrrrrrrcrrrrrr@{}}\toprule
& \multicolumn{6}{c}{Minimum number of covers} & \phantom{abc}& \multicolumn{6}{c}{Maximum number of covers} \\
\cmidrule{2-7} \cmidrule{9-14}
& $G_1$ & $G_2$ & $G_3$ & $G_4$ & $G_5$ & All && $G_1$ & $G_2$ & $G_3$ & $G_4$ & $G_5$ & All\\ \midrule
$n = 5$ & 24 & 0 & 0 & 60 & 0 & 400 && 100 & 180 & 108 & 300 & 28 & 400  \\
$n = 6$ & 0 & 288 & 0 & 0 & 0 & 1\,728 && 384 & 1\,296 & 972 & 960 & 216 & 1\,944  \\
$n = 7$ & 888 & 0 & 0 & 126 & 0 & 8\,970 && 3\,528 & 5\,220 & 2\,700 & 5\,586 & 195 & 9\,354  \\
$n = 8$ & 4\,672 & 0 & 0 & 0 & 0 & 42\,240 && 22\,016 & 29\,376 & 12\,288 & 21\,504 & 1\,536 & 48\,832  \\
\bottomrule
\end{tabular}
\caption{\label{tab:avg-q_i}The number of $(n+1)$-covers that induce $G_i$, averaged over isotopism classes of Latin square of order $n$. We also give the minimum and maximum numbers of $(n+1)$-covers found in a Latin square. The columns headed ``All'' refer to the count of all $(n+1)$-covers irrespective of which $G_i$ they induce.}
\end{table*}

Theorems~\ref{th:CovertoPT} and~\ref{th:minCoverPT} leave open some possibilities, e.g.,\ a Latin square might have two minimum covers that differ in terms of the smallest deficit of the partial transversals that they contain.  The following theorem gives some restrictions in this context.


\begin{figure}[htp]
\centering
$\begin{array}{cccc}
\begin{tikzpicture}
\matrix[square matrix]{
|[fill=blue!50]| 0 & |[fill=white]| 1 & |[fill=white]| 2 & |[fill=white]| 3 & |[fill=white]| 4 & |[fill=white]| 5 & |[fill=white]| 6 & |[fill=white]| 7 \\
|[fill=blue!50]| 1 & |[fill=white]| 2 & |[fill=white]| 3 & |[fill=white]| 4 & |[fill=white]| 5 & |[fill=white]| 6 & |[fill=white]| 7 & |[fill=white]| 0 \\
|[fill=white]| 2 & |[fill=blue!50]| 3 & |[fill=white]| 4 & |[fill=blue!50]| 5 & |[fill=white]| 6 & |[fill=white]| 7 & |[fill=white]| 0 & |[fill=white]| 1 \\
|[fill=white]| 3 & |[fill=white]| 4 & |[fill=white]| 5 & |[fill=white]| 6 & |[pattern=north east lines, pattern color=blue!50]| 7 & |[fill=white]| 0 & |[fill=white]| 1 & |[fill=white]| 2 \\
|[fill=white]| 4 & |[fill=white]| 5 & |[fill=blue!50]| 6 & |[fill=white]| 7 & |[fill=white]| 0 & |[fill=white]| 1 & |[fill=white]| 2 & |[fill=white]| 3 \\
|[fill=white]| 5 & |[fill=white]| 6 & |[fill=white]| 7 & |[fill=white]| 0 & |[fill=white]| 1 & |[pattern=north east lines, pattern color=blue!50]| 2 & |[fill=white]| 3 & |[fill=white]| 4 \\
|[fill=white]| 6 & |[fill=white]| 7 & |[fill=white]| 0 & |[fill=white]| 1 & |[fill=white]| 2 & |[fill=white]| 3 & |[pattern=north east lines, pattern color=blue!50]| 4 & |[fill=white]| 5 \\
|[fill=white]| 7 & |[fill=white]| 0 & |[fill=white]| 1 & |[fill=white]| 2 & |[fill=white]| 3 & |[fill=white]| 4 & |[fill=white]| 5 & |[fill=blue!50]| 6 \\
};
\end{tikzpicture}
&
\begin{tikzpicture}
\matrix[square matrix]{
|[fill=blue!50]| 0 & |[fill=white]| 1 & |[fill=white]| 2 & |[fill=white]| 3 & |[fill=white]| 4 & |[fill=white]| 5 & |[fill=white]| 6 & |[fill=blue!50]| 7 \\
|[fill=white]| 1 & |[fill=white]| 2 & |[fill=white]| 3 & |[fill=blue!50]| 4 & |[fill=white]| 5 & |[fill=white]| 6 & |[fill=white]| 7 & |[fill=white]| 0 \\
|[fill=white]| 2 & |[fill=white]| 3 & |[fill=white]| 4 & |[fill=white]| 5 & |[pattern=north east lines, pattern color=blue!50]| 6 & |[fill=white]| 7 & |[fill=white]| 0 & |[fill=white]| 1 \\
|[fill=white]| 3 & |[fill=blue!50]| 4 & |[fill=white]| 5 & |[fill=white]| 6 & |[fill=white]| 7 & |[fill=white]| 0 & |[fill=white]| 1 & |[fill=white]| 2 \\
|[fill=white]| 4 & |[fill=white]| 5 & |[fill=white]| 6 & |[fill=white]| 7 & |[fill=white]| 0 & |[fill=white]| 1 & |[pattern=north east lines, pattern color=blue!50]| 2 & |[fill=white]| 3 \\
|[fill=blue!50]| 5 & |[fill=white]| 6 & |[fill=white]| 7 & |[fill=white]| 0 & |[fill=white]| 1 & |[fill=white]| 2 & |[fill=white]| 3 & |[fill=white]| 4 \\
|[fill=white]| 6 & |[fill=white]| 7 & |[fill=white]| 0 & |[fill=white]| 1 & |[fill=white]| 2 & |[pattern=north east lines, pattern color=blue!50]| 3 & |[fill=white]| 4 & |[fill=white]| 5 \\
|[fill=white]| 7 & |[fill=white]| 0 & |[pattern=north east lines, pattern color=blue!50]| 1 & |[fill=white]| 2 & |[fill=white]| 3 & |[fill=white]| 4 & |[fill=white]| 5 & |[fill=white]| 6 \\
};
\end{tikzpicture}
&
\begin{tikzpicture}
\matrix[square matrix]{
|[fill=white]| 0 & |[fill=blue!50]| 1 & |[fill=white]| 2 & |[fill=white]| 3 & |[fill=blue!50]| 4 & |[fill=white]| 5 & |[fill=white]| 6 & |[fill=white]| 7 \\
|[fill=white]| 1 & |[fill=white]| 2 & |[pattern=north east lines, pattern color=blue!50]| 3 & |[fill=white]| 4 & |[fill=white]| 5 & |[fill=white]| 6 & |[fill=white]| 7 & |[fill=white]| 0 \\
|[fill=white]| 2 & |[fill=white]| 3 & |[fill=white]| 4 & |[fill=white]| 5 & |[fill=white]| 6 & |[pattern=north east lines, pattern color=blue!50]| 7 & |[fill=white]| 0 & |[fill=white]| 1 \\
|[fill=white]| 3 & |[fill=white]| 4 & |[fill=white]| 5 & |[fill=white]| 6 & |[fill=white]| 7 & |[fill=white]| 0 & |[fill=white]| 1 & |[pattern=north east lines, pattern color=blue!50]| 2 \\
|[fill=blue!50]| 4 & |[fill=white]| 5 & |[fill=white]| 6 & |[fill=white]| 7 & |[fill=white]| 0 & |[fill=white]| 1 & |[fill=white]| 2 & |[fill=white]| 3 \\
|[fill=white]| 5 & |[fill=white]| 6 & |[fill=white]| 7 & |[pattern=north east lines, pattern color=blue!50]| 0 & |[fill=white]| 1 & |[fill=white]| 2 & |[fill=white]| 3 & |[fill=white]| 4 \\
|[fill=blue!50]| 6 & |[fill=white]| 7 & |[fill=white]| 0 & |[fill=white]| 1 & |[fill=white]| 2 & |[fill=white]| 3 & |[fill=white]| 4 & |[fill=white]| 5 \\
|[fill=white]| 7 & |[fill=white]| 0 & |[fill=white]| 1 & |[fill=white]| 2 & |[fill=white]| 3 & |[fill=white]| 4 & |[pattern=north east lines, pattern color=blue!50]| 5 & |[fill=white]| 6 \\
};
\end{tikzpicture}
&
\begin{tikzpicture}
\matrix[square matrix]{
|[fill=blue!50]| 0 & |[fill=white]| 1 & |[fill=white]| 2 & |[fill=white]| 3 & |[fill=blue!50]| 4 & |[fill=white]| 5 & |[fill=white]| 6 & |[fill=white]| 7 \\
|[fill=white]| 1 & |[pattern=north east lines, pattern color=blue!50]| 2 & |[fill=white]| 3 & |[fill=white]| 4 & |[fill=white]| 5 & |[fill=white]| 6 & |[fill=white]| 7 & |[fill=white]| 0 \\
|[fill=white]| 2 & |[fill=white]| 3 & |[fill=white]| 4 & |[pattern=north east lines, pattern color=blue!50]| 5 & |[fill=white]| 6 & |[fill=white]| 7 & |[fill=white]| 0 & |[fill=white]| 1 \\
|[fill=white]| 3 & |[fill=white]| 4 & |[fill=white]| 5 & |[fill=white]| 6 & |[fill=white]| 7 & |[fill=white]| 0 & |[pattern=north east lines, pattern color=blue!50]| 1 & |[fill=white]| 2 \\
|[fill=blue!50]| 4 & |[fill=white]| 5 & |[fill=white]| 6 & |[fill=white]| 7 & |[fill=white]| 0 & |[fill=white]| 1 & |[fill=white]| 2 & |[fill=white]| 3 \\
|[fill=white]| 5 & |[fill=white]| 6 & |[pattern=north east lines, pattern color=blue!50]| 7 & |[fill=white]| 0 & |[fill=white]| 1 & |[fill=white]| 2 & |[fill=white]| 3 & |[fill=white]| 4 \\
|[fill=white]| 6 & |[fill=white]| 7 & |[fill=white]| 0 & |[fill=white]| 1 & |[fill=white]| 2 & |[pattern=north east lines, pattern color=blue!50]| 3 & |[fill=white]| 4 & |[fill=white]| 5 \\
|[fill=white]| 7 & |[fill=white]| 0 & |[fill=white]| 1 & |[fill=white]| 2 & |[fill=white]| 3 & |[fill=white]| 4 & |[fill=white]| 5 & |[pattern=north east lines, pattern color=blue!50]| 6 \\
};
\end{tikzpicture}
\\
G_1 & G_2 & G_3 & G_5
\end{array}$

\caption{\label{f:both2land2l-1}Minimum covers of the Cayley table of $\Z_8$ that induce graphs isomorphic to $G_1$, $G_2$, $G_3$ and $G_5$, respectively.  For the three right-most covers, we can delete $2$ entries from the highlighted cover to give a partial transversal of deficit $1$, but we must delete at least $3$ entries from the left-most cover to obtain a partial transversal, which will have deficit at least $2$.}
\end{figure}

\begin{figure}
\centering
\begin{tikzpicture}
\matrix[square matrix]{
|[fill=white]| 0 & |[fill=white]| 1 & |[fill=white]| 2 & |[fill=blue!50]| 3 & |[fill=white]| 4 & |[fill=white]| 5 & |[fill=white]| 6 & |[fill=white]| 7 & |[fill=white]| 8 & |[fill=white]| 9 \\
|[fill=white]| 1 & |[pattern=north east lines, pattern color=blue!50]| 2 & |[fill=white]| 3 & |[fill=white]| 4 & |[fill=white]| 5 & |[fill=white]| 6 & |[fill=white]| 7 & |[fill=white]| 8 & |[fill=white]| 9 & |[fill=white]| 0 \\
|[fill=white]| 2 & |[fill=white]| 3 & |[fill=white]| 4 & |[fill=white]| 5 & |[fill=white]| 6 & |[fill=white]| 7 & |[fill=white]| 8 & |[fill=white]| 9 & |[fill=white]| 0 & |[fill=blue!50]| 1 \\
|[fill=white]| 3 & |[fill=white]| 4 & |[fill=white]| 5 & |[fill=white]| 6 & |[pattern=north east lines, pattern color=blue!50]| 7 & |[fill=white]| 8 & |[fill=white]| 9 & |[fill=white]| 0 & |[fill=white]| 1 & |[fill=white]| 2 \\
|[fill=white]| 4 & |[fill=white]| 5 & |[fill=white]| 6 & |[fill=white]| 7 & |[fill=white]| 8 & |[fill=white]| 9 & |[pattern=north east lines, pattern color=blue!50]| 0 & |[fill=white]| 1 & |[fill=white]| 2 & |[fill=white]| 3 \\
|[fill=white]| 5 & |[fill=white]| 6 & |[fill=white]| 7 & |[fill=white]| 8 & |[fill=white]| 9 & |[fill=white]| 0 & |[fill=white]| 1 & |[fill=white]| 2 & |[fill=blue!50]| 3 & |[fill=white]| 4 \\
|[fill=white]| 6 & |[fill=white]| 7 & |[pattern=north east lines, pattern color=blue!50]| 8 & |[fill=white]| 9 & |[fill=white]| 0 & |[fill=white]| 1 & |[fill=white]| 2 & |[fill=white]| 3 & |[fill=white]| 4 & |[fill=white]| 5 \\
|[fill=white]| 7 & |[fill=white]| 8 & |[fill=white]| 9 & |[fill=white]| 0 & |[fill=white]| 1 & |[fill=white]| 2 & |[fill=white]| 3 & |[fill=white]| 4 & |[fill=white]| 5 & |[fill=blue!50]| 6 \\
|[fill=white]| 8 & |[fill=white]| 9 & |[fill=white]| 0 & |[fill=white]| 1 & |[fill=white]| 2 & |[fill=white]| 3 & |[fill=white]| 4 & |[pattern=north east lines, pattern color=blue!50]| 5 & |[fill=white]| 6 & |[fill=white]| 7 \\
|[fill=blue!50]| 9 & |[fill=white]| 0 & |[fill=white]| 1 & |[fill=white]| 2 & |[fill=white]| 3 & |[fill=blue!50]| 4 & |[fill=white]| 5 & |[fill=white]| 6 & |[fill=white]| 7 & |[fill=white]| 8 \\
};
\end{tikzpicture}
\caption{\label{fi:Z10cover}A minimum cover $\cov$ of the Cayley table of $\mathbb{Z}_{10}$, where $\cov$ does not contain a partial transversal of deficit $1$, and each partial transversal of deficit $2$ in $\cov$ is maximal.}
\end{figure}

\begin{theo}\label{t:partransinmincov}
Let $L$ be a Latin square of order $n\ge2$ in which the minimum deficit of a partial transversal is $d$.  Then:
\begin{enumerate}
 \item the minimum size of a cover of $L$ is $n+\lceil d/2 \rceil$,
 \item any partial transversal $T$ of deficit $d$ is contained in a minimum cover of $L$,
 \item any minimum cover contains a partial transversal of deficit $d$ if $d$ is even, or deficit $d+1$ if $d$ is odd, and
 \item if $d$ is even, any minimum cover that contains an entry $\mathbf{e}$ contains a minimum-deficit partial transversal that contains $\mathbf{e}$.
\end{enumerate}
\end{theo}

\begin{proof}
If $L$ has a cover of size less than $n+\lceil d/2 \rceil$, then \tref{th:CovertoPT} implies it has a partial transversal of deficit less than $d$, contradicting the assumption that $d$ is the minimum deficit.  So any cover has size at least $n+\lceil d/2 \rceil$, and \tref{th:PTtoCover} implies $T$ is contained in some cover of size $n+\lceil d/2 \rceil$.

\tref{th:minCoverPT} implies that any $(n+\lceil d/2 \rceil)$-cover contains a partial transversal of deficit $2\lceil d/2 \rceil$ or $2\lceil d/2 \rceil-1$.  When $d$ is odd, these equal $d+1$ and $d$, respectively, and we can delete an entry from a partial transversal of deficit $d$ to obtain one of deficit $d+1$.  When $d$ is even, \tref{th:CovertoPT} implies the cover contains a partial transversal of deficit $d$, and we can ensure $\mathbf{e}$ belongs to this partial transversal by choosing $\mathbf{e} \in R \cap C \cap S$ in the proof of Theorem~\ref{th:CovertoPT}.
\end{proof}

\tref{t:partransinmincov} implies nothing consequential when $d=0$.  Cyclic groups of even order have minimum deficit $d=1$, and are thus a convenient example to verify that the conditions of \tref{t:partransinmincov} cannot be tightened.  To date, we only have examples of Latin squares where the minimum deficit of a partial transversal is $d \in \{0,1\}$, with Brualdi's Conjecture implying that this is always the case, so we cannot inspect $d \geq 2$ cases.  By inspecting cyclic group tables of even order, we make the following observations:

\begin{itemize}
 \item A minimum cover might not contain a partial transversal of minimum deficit $d$, but instead have one of deficit $d+1$.  \fref{f:both2land2l-1} gives an example of this; we give four covers of the Cayley table of $\Z_8$, one of which contains no partial transversal of deficit $1$.
 \item A minimum cover might contain two maximal partial transversals, one of deficit $d$ and one of deficit $d+1$.  The cover that induces $G_2$ in \fref{f:both2land2l-1} has this property.
 \item A minimum cover might contain no partial transversals of minimum deficit $d$, with the partial transversals of deficit $d+1$ it contains all being maximal.  For the the Cayley table of $\mathbb{Z}_{10}$, \fref{fi:Z10cover} shows an $11$-cover that contains no partial transversals of deficit $1$, and the $8$ partial transversals of deficit $2$ it contains are maximal.
\end{itemize}

Given the five possible graph structures of $(n+1)$-covers in \fref{fig:subgraphs}, we can enumerate the number of partial transversals of deficit $d$ they contain, which we do in Table~\ref{tab:nplus1PT}.  The terms in Table~\ref{tab:nplus1PT} are derived as follows: For each way we can delete $b+1$ vertices from the graph $G_i$ to form an edgeless graph, we can form a partial transversal of deficit $d$ by deleting them along with a further $d-b$ entries that are not involved in $G_i$.  Each partial transversal of deficit $d$ generated this way is distinct.  Generally, these are not maximal partial transversals, but they are maximal partial transversals when $d=0$, or when $d=1$ for graphs other than $G_4$.

\begin{table*}[ht]\centering
\begin{tabular}{@{}rlr@{\hskip5pt}c@{\hskip5pt}r@{\hskip5pt}c@{\hskip5pt}r@{\hskip5pt}c@{\hskip5pt}r@{\hskip5pt}llccc@{}}\toprule
& \multicolumn{9}{c}{Number of partial transversals of deficit $d$} & \phantom{abc}& \multicolumn{3}{c}{$d$} \\ \cmidrule{12-14}
& \multicolumn{9}{c}{contained in an $(n+1)$-cover that induces $G_i$} & & $0$ & $1$ & $2$ \\
\midrule
$G_1$ && $8\binom{n-5}{d-2}$ & $+$ & $12\binom{n-5}{d-3}$ & $+$ & $6\binom{n-5}{d-4}$ & $+$ & $\binom{n-5}{d-5}$ &&& $-$ & $-$ & $8$ \\
$G_2$ && $2\binom{n-4}{d-1}$ & $+$ & $7\binom{n-4}{d-2}$ & $+$ & $5\binom{n-4}{d-3}$ & $+$ & $\binom{n-4}{d-4}$ &&& $-$ & $2$ & $2n-1$ \\
$G_3$ &&                     &     & $3\binom{n-3}{d-1}$ & $+$ & $4\binom{n-3}{d-2}$ & $+$ & $\binom{n-3}{d-3}$ &&& $-$ & $3$ & $3n-5$ \\
$G_4$ && $\binom{n-3}{d}$    & $+$ & $3\binom{n-3}{d-1}$ & $+$ & $4\binom{n-3}{d-2}$ & $+$ & $\binom{n-3}{d-3}$ &&& $1$ & $n+3$ & $\tfrac{1}{2} (n^2-n+2)$ \\
$G_5$ &&                     &     &                     &     & $3\binom{n-2}{d-1}$ & $+$ & $\binom{n-2}{d-2}$ &&& $-$ & $3$ & $3n-5$ \\
\bottomrule
\end{tabular}
\caption{\label{tab:nplus1PT} The number of distinct partial transverals of deficit $d$ within an $(n+1)$-cover that induces the subgraph $G_i$.}
\end{table*}

\tabref{tab:nplus1PT} shows that the number of partial transversals of deficit $d$ in an $(n+1)$-cover depends significantly on its structure, e.g., when $d=2$, the number varies from $\Theta(1)$ to $\Theta(n^2)$.  We therefore do not anticipate a simple relationship between the number of partial transversals and the number of $(n+1)$-covers in general.  However, we have the following results for $(n+1)$-covers.

\begin{theo}\label{th:nplus1}
Let $T$ be a maximal partial transversal of deficit $d$ in a Latin square $L$ of order $n \geq 3$.
\begin{enumerate}
 \item If $d=0$, then $T$ belongs to exactly $n^2-n$ distinct $(n+1)$-covers, none of which are minimal.
 \item If $d=1$, then $T$ belongs to exactly $3(n-1)$ distinct $(n+1)$-covers, each of which is minimal.
 \item If $d=2$, then $T$ belongs to exactly $8$ distinct $(n+1)$-covers.
 \item If $d\ge 3$, then $T$ does not belong to any $(n+1)$-cover.
\end{enumerate}
\end{theo}

\begin{proof}
The $d=0$ case is trivial, so we begin with the case $d=1$.  Assume row $i$, column $j$ and symbol $k$ are unrepresented by $T$.  Since $T$ is maximal, $(i,j,k) \not\in E(L)$, so to extend it to an $(n+1)$-cover, we must add entries of the form
\begin{itemize}
 \item $(i,j,\unkn)$ and $(\unkn,\unkn,k)$,
 \item $(i,\unkn,k)$ and $(\unkn,j,\unkn)$, excluding $(i,j,\unkn)$, or
 \item $(\unkn,j,k)$ and $(i,\unkn,\unkn)$, excluding $(i,j,\unkn)$ and $(i,\unkn,k)$.
\end{itemize}
This gives $3n-3$ distinct ways to extend $T$ to an $(n+1)$-cover.  Each cover induces a graph of type $G_2$, $G_3$ or $G_5$ (cf.\ \fref{fig:subgraphs}).  In particular, as $G_4$ does not arise, the $(n+1)$-covers are minimal.

Now assume $d=2$.  Assume rows $i$ and $i'$, columns $j$ and $j'$ and symbols $k$ and $k'$ are unrepresented by $T$.  Since $T$ is maximal, there are no entries of the form $(x,y,z)$ with $x \in \{i,i'\}$, $y \in \{j,j'\}$ and $z \in \{k,k'\}$.  One $(n+1)$-cover has the form $T \cup \{(i,j,\unkn),(i',\unkn,k),(\unkn,j',k')\}$, and we obtain all others by some combination of swapping $i$ and $i'$, swapping $j$ and $j'$, and/or swapping $k$ and $k'$.

When $d \geq 3$, Theorem~\ref{th:PTtoCover} implies that $T$ does not extend to an $(n+1)$-cover.
\end{proof}

In the $d=2$ case of \tref{th:nplus1}, the $8$ distinct $(n+1)$-covers may or may not be minimal depending on the structure of $L$.  For example, when $n=3$, they are all non-minimal (since Latin squares of order $3$ have no minimal $4$-covers).

\begin{theo}\label{th:pmaxbound}
Let $L$ be a Latin square of order $n \geq 3$. Let $\pmax$ be the number of maximal partial transversals of deficit $1$ in $L$.  Let $\qmin$ be the number of minimal $(n+1)$-covers in $L$.  Then $\qmin = q_1 + q_2 + q_3 + q_5$ and
\begin{equation}\label{e:pmaxbnds}
0 \leq \frac{2(\qmin-q_1)}{3n-4} \leq \pmax \leq \frac{2(\qmin-q_1)}{3n-6} \leq \frac{2\qmin}{3n-6}.
\end{equation}
If $t$ is the number of transversals in $L$, then the number $p$ of (not necessarily maximal) partial transversals of deficit $1$ of $L$ and the number $q$ of (not necessarily minimal) $(n+1)$-covers of $L$ satisfy
\begin{equation}\label{e:pqrel}
p \leq \frac{2q+n(n-4)t}{3n-6}.
\end{equation}
\end{theo}


\begin{proof}
The theorem is easily checked when $n=3$ since $\pmax=\qmin=q_1=0$ in this case, so assume $n \geq 4$. \tref{th:nplus1} implies that each maximal partial transversal $T$ of deficit $1$ embeds in exactly $3(n-1)$ distinct minimal $(n+1)$-covers.  Moreover, in the proof of \tref{th:nplus1}, we observed that these $(n+1)$-covers are of type $G_2$, $G_3$ or $G_5$, which contain exactly $2$, $3$ and $3$ maximal partial transversals of deficit $1$, respectively.  Thus,
\begin{equation}\label{eq:pmax}
3(n-1)\pmax = 2q_2 + 3q_3 + 3q_5.
\end{equation}
We also know
\[\qmin = q_1 + q_2 + q_3 + q_5\]
since $(n+1)$-covers are minimal if and only if they do not induce $G_4$.  Hence
\begin{equation}\label{eq:pmaxbound}
3(n-1)\pmax = 2\qmin -2q_1 + q_3 + q_5.
\end{equation}

Let $T$ be a maximal partial transversal of $L$ of deficit $1$. Up to isotopism of $L$, we may assume that $T=\{(i,i,i) : i \in \mathbb{Z}_n \setminus \{z\}\}$, where $z=n-1$, and $L_{zz}=0$.  Define $r$ such that $L_{rz}=z$ and define $c$ such that $L_{zc}=z$.  Among the $3(n-1)$ distinct minimal $(n+1)$-covers containing $T$, we have the following three families:
\begin{align*}
&\big\{T \cup \{(z,z,0),(i,j,z)\} : i,j \in \mathbb{Z}_n \setminus \{0,z\}  \big\},\\
&\big\{T \cup \{(r,z,z),(z,j,k)\} : j,k \in \mathbb{Z}_n \setminus\{r,z\}  \big\}, \text{ and}\\
&\big\{T \cup \{(z,c,z),(i,z,k)\} : i,k \in \mathbb{Z}_n \setminus \{c,z\}  \big\}. 
\end{align*}
Each family accounts for at least $n-4$ distinct minimal $(n+1)$-covers containing $T$ and inducing $G_2$.  Since there are $3(n-1)$ minimal $(n+1)$-covers containing $T$, there can be at most $9$ that do not induce $G_2$, and hence either induce $G_3$ or $G_5$. 
We note that $T$ is contained in at least $3$ distinct minimal $(n+1)$-covers that do not induce $G_2$, corresponding to the three choices of two entries from $\{(r,z,z),(z,c,z),(z,z,0)\}$. 
This means there are between $3$ and $9$ distinct $(n+1)$-covers that induce $G_3$ or $G_5$ and contain $T$. 
Also, recall that each $(n+1)$-cover that induces $G_3$ or $G_5$ contains exactly $3$ maximal partial transversals of deficit $1$. 
This gives $ 3\pmax \leq 3q_3 + 3q_5 \leq 9\pmax$ or simply $\pmax \leq q_3 + q_5 \leq 3 \pmax$, which we substitute into \eqref{eq:pmaxbound} to obtain \eqref{e:pmaxbnds}.

The number $p$ of (not necessarily maximal) partial transversals of deficit $1$ of $L$ is $p=\pmax+nt$ and the number $q$ of (not necessarily minimal) $(n+1)$-covers of $L$ is $q=\qmin+q_4=\qmin+n(n-1)t$.  Combining this with \eref{e:pmaxbnds}, we get \eqref{e:pqrel}. 
\end{proof}

Our next result is motivated by the work of Belyavskaya and Russu (see
\cite[p.\,179]{DKII}) who showed that Cayley tables of certain groups do
not have maximal partial transversals of deficit $1$, in which case
$\pmax=q_2=q_3=q_5=0$. This is an obstacle to finding a non-trivial
lower bound on $\pmax$ that is only a function of $\qmin$ and $n$.

\begin{lemm}\label{l:grpG1-5}
  Let $L$ be the Cayley table of an abelian group $\G$ of order $n$. If
  the Sylow $2$-subgroups of $\G$ are trivial or non-cyclic then $L$
  has no maximal partial transversal of deficit $1$ (and hence has no
  $(n+1)$-cover inducing $G_2$, $G_3$ or $G_5$). On the other hand,
  if the Sylow $2$-subgroups of $\G$ are non-trivial and cyclic then $L$
  has no transversal (and hence has no $(n+1)$-cover inducing $G_4$).
\end{lemm}

\begin{proof}
  Let $X_\G$ denote the sum of the elements of $\G$. It is well-known
  (see, for example, \cite[p.\,9]{DKII}) that $X_\G$ is the identity if
  the Sylow $2$-subgroups of $\G$ are trivial or non-cyclic and is
  otherwise equal to the unique element of order $2$ in $\G$. In the
  latter case there are no transversals in $L$ (\cite[p.\,8]{DKII}) as
  claimed, so we concentrate on the former case. Suppose $T$ is a
  partial transversal of deficit $1$ in $L$ and that $r$, $c$ and $s$ are
  respectively the row, column and symbol that are not represented in
  $T$. Then $-s=X_\G-s=(X_\G-r)+(X_\G-c)=-r-c$ because $L$ is the Cayley
  table of $\G$. As $s=r+c$, we see immediately that $T$ is not
  maximal, from which the result follows.
\end{proof}

\tabref{tab:Zn-q_i} gives the value of $q_i$ for the Cayley table of $\mathbb{Z}_n$.  The zeroes in \tabref{tab:Zn-q_i} are all explained by \lref{l:grpG1-5}, except that $q_5=0$ in $\Z_6$, which may just be a small order quirk.

\begin{table*}[ht]\centering
\begin{tabular}{@{}lrrrrrr@{}}\toprule
& $q_1$ & $q_2$ & $q_3$ & $q_4$ & $q_5$ & All\\ \midrule
$\mathbb{Z}_5$ & 100 & 0 & 0 & 300 & 0 & 400 \\
$\mathbb{Z}_6$ & 144 & 864 & 864 & 0 & 0 & 1\,872\\
$\mathbb{Z}_7$ & 3528 & 0 & 0 & 5\,586 & 0 & 9\,114 \\
$\mathbb{Z}_8$ & 7424 & 27\,648 & 9\,216 & 0 & 1\,024 & 45\,312 \\
$\mathbb{Z}_9$ & 115\,668 & 0 & 0 & 145\,800 & 0 & 261\,468 \\
$\mathbb{Z}_{10}$ & 326\,400 & 864\,000 & 249\,600 & 0 & 9\,600 & 1\,449\,600 \\ 
$\mathbb{Z}_{11}$ & 4\,692\,380 & 0 & 0 & 4\,163\,610 & 0 & 8\,855\,990 \\
\bottomrule
\end{tabular}
\caption{\label{tab:Zn-q_i}The number $q_i$ of $(n+1)$-covers of $\mathbb{Z}_n$ that induce $G_i$.}
\end{table*}

The maximal partial transversal highlighted in the Latin square
\begin{center}
\begin{tikzpicture}
\matrix[square matrix]{
|[fill=blue!50]| 0 & |[fill=white]| 3 & |[fill=white]| 4 & |[fill=white]| 5 & |[fill=white]| 2 & |[fill=white]| 1 \\
|[fill=white]| 3 & |[fill=blue!50]| 1 & |[fill=white]| 0 & |[fill=white]| 4 & |[fill=white]| 5 & |[fill=white]| 2 \\
|[fill=white]| 5 & |[fill=white]| 4 & |[fill=blue!50]| 2 & |[fill=white]| 0 & |[fill=white]| 1 & |[fill=white]| 3 \\
|[fill=white]| 4 & |[fill=white]| 2 & |[fill=white]| 1 & |[fill=blue!50]| 3 & |[fill=white]| 0 & |[fill=white]| 5 \\
|[fill=white]| 2 & |[fill=white]| 5 & |[fill=white]| 3 & |[fill=white]| 1 & |[fill=blue!50]| 4 & |[fill=white]| 0 \\
|[fill=white]| 1 & |[fill=white]| 0 & |[fill=white]| 5 & |[fill=white]| 2 & |[fill=white]| 3 & |[fill=white]| 4 \\
};
\end{tikzpicture}
\end{center}
is contained in exactly $9$ (necessarily minimal) $(n+1)$-covers that do not induce $G_2$.  We also saw during the proof of \tref{th:pmaxbound} that all maximal partial transversals of deficit $1$ are contained in at least $3$ (necessarily minimal) $(n+1)$-covers that do not induce $G_2$. 
These observations present some obstacles to improving the bounds given in \eref{e:pmaxbnds}.

We also observe that in a general Latin square, \eqref{eq:pmax} implies that $q_2 \equiv 0 \pmod 3$. Akbari and Alipour~\cite{AkbariAlipour} showed that $\pmax \equiv 0 \pmod 4$. Thus, $2q_2 \equiv q_3+q_5 \pmod 4$ by \eqref{eq:pmax}.  We also note $q_4=n(n-1)t$, where $t$ is the number of transversals, and $t$ is even when the order $n$ is even \cite{Balasubramanian1990}.  Also, we can delete an entry from any $(n+1)$-cover of type $G_3$ (when $n \geq 5$) or type $G_5$ (when $n \geq 4$) and add another entry to obtain an $(n+1)$-cover of type $G_2$, implying that if $q_2=0$, then $q_3=q_5=0$ (which occurs for the odd-order cyclic group tables).


The question of which entries within Latin squares belong to transversals has also been studied. The parallel topic for covers plays a role throughout this paper, so we mention the following theorem, which can be derived from \cite{EW12} and \cite{WW06}.

\begin{theo}\label{t:notinmin}
For every $n\ge5$, there exists a Latin square of order $n$ that has transversals, but also has an entry that is not in any transversal.  Consequently, for every $n\ge5$, there exists a Latin square of order $n$ that contains an entry that is not in any minimum cover nor in any partial transversal of minimum deficit.
\end{theo}




\tref{t:notinmin} does not extend to any order $n\le4$ since all Latin squares of those orders are isotopic to the Cayley table of a group. Such Latin squares have autotopism groups that act transitively on entries, and hence every entry will be in a partial transversal of minimum deficit and also every entry will be in a minimum cover. 

Consider a Latin square of $L$ in which the minimum deficit of a partial transversal is $d$. By \tref{t:partransinmincov}, every entry of $L$ that is in a partial transversal of deficit $d$ is also in a minimum cover. It is not clear if the converse holds when $d$ is odd (although \tref{t:partransinmincov} shows the converse holds when $d$ is even).  There is no known Latin square of order $n \geq 2$ that has an entry that is not in a partial transversal of deficit $1$. If this property holds in general, then every entry that is in a minimum cover is also in a partial transversal of minimum deficit.

To finish this section, we observe that if a Latin square $L$ of order $n \geq 5$ has a transversal, then any entry in $L$ belongs to a minimal $(n+1)$-cover.  Therefore, the entries in \tref{t:notinmin} that are not in minimum covers do belong to minimal covers of size 1 larger than minimum.

\begin{theo}\label{th:nplusone}
If a Latin square $L$ of order $n \geq 5$ has a transversal $T$, then each entry of $L$ belongs to some minimal $(n+1)$-cover.
\end{theo}

\begin{proof}
A computer search reveals that any entry in any Latin square of order $n \in \{5,6\}$ belongs to a minimal $(n+1)$-cover.  Now assume $n \geq 7$.  By applying an isotopism, we may assume that $T=\{(i,i,i) : i \in \mathbb{Z}_n\}$. Let $(a,b,c)$ be an arbitrary entry of $L$.

First, we consider the case when no transversal contains $(a,b,c)$ (implying that $a \neq b$). Consider 
\[\cov= \left( T \setminus \{(a,a,a),(b,b,b)\} \right) \cup \{(a,b,c),(b,c',a),(c'',a,b)\}.
\]
Now, $\cov$ is a clearly a cover of $L$, and is minimal unless $c = c' = c''$ leaving $(c,c,c)$ as a redundant entry. However, if $c = c' = c''$, then $\cov \setminus \{(c,c,c)\}$ would be a transversal containing $(a,b,c)$, which we assumed did not exist, so $\cov$ must be a minimal $(n+1)$-cover containing $(a,b,c)$.

Now we may assume that $(a,b,c) \in T$ and that $a = b = c = 0$. Let $i$ be such that $i \not\in \{0,1\}$ and $L_{1i} \neq 0$. Consider
$$\cov_1 = \left( T \setminus \{(1,1,1),(i,i,i)\} \right) \cup \{(1,i,j),(i,j',1),(j'',1,i)\}.$$
By a similar argument as before, if $j$, $j'$ and $j''$ are not all the same, then $\cov_1$ is a minimal $(n+1)$-cover containing $(a,b,c)$. If $j = j' = j''$, then note that $j \not\in \{0,1,i\}$, and then let $k$ be such that $k \not\in \{0,1,i,j\}$ and $L_{1k} \not\in \{0,i\}$ (this choice of $k$ is possible since $n \geq 7$). Consider
$$\cov_2 = \left( T \setminus \{(1,1,1),(k,k,k)\} \right) \cup \{(1,k,\ell),(k,\ell',1),(\ell'',1,k)\}.$$
By a similar argument as before, if $\ell$, $\ell'$ and $\ell''$ are not all the same, then $\cov_2$ is a minimal $(n+1)$-cover containing $(a,b,c)$. If $\ell = \ell' = \ell''$, then note that $\ell \not\in \{0,1,j,k\}$ and $L$ must have the following structure:
\[
\begin{array}{c}
\begin{tikzpicture}
\fill[blue!50] (-3*13pt,2*15pt) -- (-3*13pt,3*15pt) -- (-2*13pt,2*15pt) -- cycle;
\fill[pattern=crosshatch, pattern color=red!50]  (-2*13pt,3*15pt) -- (-3*13pt,3*15pt) -- (-2*13pt,2*15pt) -- cycle;

\matrix[square matrix]{
$0$ &|[fill=white  ]|     &|[fill=white  ]|     &|[fill=white]|     &|[fill=white]|     &|[fill=white]|  \\
|[fill=white]|       &|[pattern=crosshatch, pattern color=red!50]| $1$ &|[fill=blue!50]| $j$ &|[fill=white]|     &|[fill=blue!50]| $\ell$ &|[fill=white]| \\
|[fill=white]|       &|[fill=white  ]|     &|[pattern=crosshatch, pattern color=red!50]| $i$ &|[fill=blue!50]| $1$ &|[fill=white]|     &|[fill=white]| \\
|[fill=white]|       &|[fill=blue!50]| $i$ &|[fill=white  ]|     &|[pattern=crosshatch, pattern color=red!50]| $j$ &|[fill=white]|     &|[fill=white]| \\
|[fill=white]|       &|[fill=white  ]|     &|[fill=white  ]|     &|[fill=white]|     &|[pattern=crosshatch, pattern color=red!50]| $k$ &|[fill=blue!50]| $1$ \\
|[fill=white]|       &|[fill=blue!50]| $k$ &|[fill=white  ]|     &|[fill=white]|     &|[fill=white]|     &|[pattern=crosshatch, pattern color=red!50]| $\ell$ \\
};
\end{tikzpicture}
\end{array}.
\]
By removing the entries containing the symbols $1,i,j,k$ and $\ell$ from $T$ and adding the shaded entries, we have a minimal $(n+1)$-cover containing $(a,b,c)$.
\end{proof}

\tref{th:nplusone} does not hold for orders $n \in \{1,3,4\}$ as the Latin squares of those orders that have transversals do not have minimal $(n+1)$-covers (and \tref{th:nplusone} is vacuously true when $n=2$).

\section{Large minimal covers}\label{s:maxmincov}

In this section, we consider the question of how large a minimal cover in a Latin square of order $n$ can be.

When $n \geq 3$, a transversal (which has size $n$) is the smallest minimal cover possible in a Latin square of order $n$. The size of the largest minimal cover is harder to establish. It is clear that it cannot be larger than size $3n$, since there are only $3n$ lines and each entry in a minimal cover uniquely represents at least one line. Perhaps surprisingly, this is close to the true answer. We will show that every Latin square of order $n$ has a minimal cover with size asymptotically 
equal to $3n$ as $n\rightarrow\infty$.


To work towards finding the size of the largest minimal covers, we begin with a simple observation.

\begin{lemm}\label{l:trivcov}
Every Latin square $L$ of order $n \geq 1$ contains a minimal cover of size $2n-1$.  Furthermore, any entry of $L$ belongs to a minimal cover of size $2n-1$.
\end{lemm}

\begin{proof}
Take all entries that are in the $r$-th row and/or in the $c$-th column.  This gives a set of $2n-1$ entries in which every line is represented.  The entry $(r,c,L_{rc})$ uniquely represents its symbol. The other entries in row $r$ uniquely represent their respective columns, and the other entries in column $c$ uniquely represent their respective rows. Hence the cover is minimal.
\end{proof}

We consider a more general problem that will be easier to deal with.  If an $n \times n$ partial Latin square on the symbol set $\mathbb{Z}_n$ has each row, column and symbol represented at least once, we call it a \emph{potential cover} of order $n$.  By definition, a cover admits a completion to a Latin square, whereas not all potential covers admit a completion.  \fref{fi:cover_potent} gives an example of two potential covers, one of which is a cover.  A potential cover $\cov$ of order $n$ is \emph{minimal} if, for all $\mathbf{e} \in \cov$, the set $\cov \setminus \{\mathbf{e}\}$ is not a potential cover of order $n$.  We will bound the maximum size of minimal potential covers, thereby giving an upper bound on the cardinality of minimal covers.


\begin{figure}[htp]
\centering
$\begin{array}{ccc}
\begin{array}{c}
\begin{tikzpicture}
\matrix[square matrix]{
|[fill=blue!50]| 0 & |[fill=white]| & |[fill=white]| & |[fill=white]| \\
|[fill=white]| & |[fill=blue!50]| 1 & |[fill=white]| & |[fill=white]| \\
|[fill=white]| & |[fill=white]| & |[fill=blue!50]| 1 & |[fill=blue!50]| 2 \\
|[fill=white]| & |[fill=white]| & |[fill=blue!50]| 3 & |[fill=blue!50]| 1 \\
};
\end{tikzpicture}
\end{array}
&
&
\begin{array}{c}
\begin{tikzpicture}
\matrix[square matrix]{
|[fill=blue!50]| 0 & |[fill=white]| 3 & |[fill=white]| 2 & |[fill=blue!50]| 1 \\
|[fill=white]| 2 & |[fill=blue!50]| 1 & |[fill=white]| 0 & |[fill=white]| 3 \\
|[fill=white]| 3 & |[fill=white]| 0 & |[fill=blue!50]| 1 & |[fill=blue!50]| 2 \\
|[fill=white]| 1 & |[fill=white]| 2 & |[fill=blue!50]| 3 & |[fill=white]| 0 \\
};
\end{tikzpicture}
\end{array}
\end{array}$
\caption{\label{fi:cover_potent}Two potential covers of Latin squares of order $4$.  Only the right potential cover admits a completion to a Latin square of order $4$ (as indicated) and is therefore a cover.}
\end{figure}

Given a potential cover $\cov$, define $\Ur=\Ur(\cov)$ to be the set of all entries that uniquely represent a row but no other line, $\Urc=\Urc(\cov)$ to be the set of all entries that uniquely represent a row and a column but no other line, $\Urcs=\Urcs(\cov)$ to be the set of all entries that uniquely represent a row, column and symbol, and define $\Uc$, $\Us$, $\Urs$ and $\Ucs$ accordingly.  An example is given in \fref{fi:RCSexample}.

If an entry does not uniquely represent a row, column or symbol, then it can be deleted to give a smaller potential cover, i.e., the potential cover is not minimal.  If a potential cover $\cov$ is minimal, then $\{\Ur, \Uc, \Us, \Urc, \Urs, \Ucs,\Urcs\}$ is a partition of $\cov$.


\begin{figure}[htp]
\centering
$\begin{array}{ccc}
\begin{array}{c}
\begin{tikzpicture}
\matrix[square matrix]{
|[fill=blue!50]| 2 & |[fill=white]| 6 & |[fill=white]| 0 & |[fill=white]| 5 & |[fill=white]| 3 & |[fill=white]| 4 & |[fill=white]| 1 \\
|[fill=white]| 6 & |[fill=blue!50]| 4 & |[fill=white]| 2 & |[fill=white]| 1 & |[fill=white]| 0 & |[fill=white]| 3 & |[fill=white]| 5 \\
|[fill=white]| 1 & |[fill=white]| 2 & |[fill=blue!50]| 3 & |[fill=white]| 0 & |[fill=white]| 6 & |[fill=white]| 5 & |[fill=blue!50]| 4 \\
|[fill=white]| 0 & |[fill=white]| 3 & |[fill=white]| 4 & |[fill=blue!50]| 6 & |[fill=white]| 5 & |[fill=white]| 1 & |[fill=white]| 2 \\
|[fill=white]| 3 & |[fill=blue!50]| 5 & |[fill=white]| 1 & |[fill=white]| 2 & |[fill=white]| 4 & |[fill=white]| 6 & |[fill=white]| 0 \\
|[fill=white]| 5 & |[fill=white]| 0 & |[fill=white]| 6 & |[fill=white]| 4 & |[fill=white]| 1 & |[fill=blue!50]| 2 & |[fill=white]| 3 \\
|[fill=white]| 4 & |[fill=blue!50]| 1 & |[fill=white]| 5 & |[fill=white]| 3 & |[fill=white]| 2 & |[fill=blue!50]| 0 & |[fill=white]| 6 \\
};
\end{tikzpicture}
\end{array}
&
&
\begin{array}{rl}
\Ur= & \{(1,1,4),(5,5,2)\} \\
\Uc= & \{(2,6,4)\} \\
\Us= & \{(6,1,1),(6,5,0)\} \\
\Urc= & \{(0,0,2)\} \\
\Urs= & \{(4,1,5)\} \\
\Ucs= & \{(2,2,3)\} \\
\Urcs= & \{(3,3,6)\} \\
\end{array}
\end{array}$
\caption{\label{fi:RCSexample}Illustrating $\Ur$, $\Urc$, etc.\ for a minimal cover of a Latin square of order $7$.}   
\end{figure}

Throughout the next proof, we edit a minimal potential cover $\cov$ by deleting a few entries from it, and adding others, which creates a modified minimal potential cover.  After such edits, to verify the result is indeed a minimal potential cover, we need to check the following three properties:
\begin{enumerate}
 \item \textit{Partial Latin square}.  When adding entries, we must ensure we do not violate the partial Latin square property by adding an entry to an already filled cell, or by adding a symbol to a row or column that already contains that symbol.
 
 \item \textit{Potential cover}.  After deleting entries from a minimal potential cover, we necessarily end up with some rows, columns and/or symbols unrepresented.  These rows, columns and/or symbols must be represented by newly added entries.
 
 \item \textit{Minimality}.  We need to verify that each entry uniquely represents some row, column or symbol.  We need only check this for the newly added entries and any entries that share a row, column or symbol with a newly added entry. This last point is the most subtle: it is easy to overlook that adding an entry might make another entry redundant.
\end{enumerate}
We omit details of such routine checks without further comment.

\begin{lemm}\label{lm:RCSbounds}
Let $n \geq 2$. There exists a minimal potential cover $M$ of order $n$, which is at least as large as all other minimal potential covers of order $n$, and has the following additional properties
\begin{align*}
\Urc&=\Urs=\Ucs=\Urcs=\emptyset,\\
|M| &= |\Ur|+|\Uc|+|\Us|, \\
0<|\Ur| &\leq (n-|\Uc|)(n-|\Us|),\\
0<|\Uc| &\leq (n-|\Ur|)(n-|\Us|), \text{ and}\\
0<|\Us| &\leq (n-|\Ur|)(n-|\Uc|).
\end{align*}
\end{lemm}

\begin{proof} 
We assume that $\cov$ is some minimal potential cover of order $n$ of the largest possible size. If $|\cov|\le 2n-1$, then the cover described in the proof of \lref{l:trivcov} satisfies the required conditions. So we may assume that $|\cov|\ge2n$ (which implies that $n\ge3$).

We first argue that $\Urcs=\emptyset$.  If $(r,c,s) \in \Urcs$, then since $|\cov| \geq 2n$ there  is a row $r' \neq r$ that contains at least two entries in $\cov$ and, similarly, there is some column $c' \neq c$ that contains at least two entries in $\cov$. But then
\begin{equation*}\label{eq:CmodRCS}
\big(\cov \setminus \{(r,c,s)\}\big) \cup \big\{(r',c,s),(r,c',s)\big\}
\end{equation*}
is a larger potential cover, contradicting the choice of $\cov$.  So $\Urcs=\emptyset$.

Next we explain how we can modify $\cov$ in such a way that $\Us$ and/or $\Urc$ becomes empty, without decreasing the size of $\cov$ nor violating the minimal potential cover property.

Suppose that there exists $(r_0,c_0,s_0) \in \Us$ and
$(r_1,c_1,s_1) \in \Urc$. Note that the entries
$(r_0,c_0,s_0)$ and $(r_1,c_1,s_1)$ cannot agree in any coordinate.  We now split into three cases.

\medskip\noindent
\textbf{Case I:} \
Symbol $s_1$ does not appear in row $r_0$ nor in column $c_0$. 

In this case,
\begin{equation*}\label{eq:Cmod1}
\big(\cov \setminus \{(r_1,c_1,s_1)\}\big) \cup \big\{(r_0,c_1,s_1), (r_1,c_0,s_1)\big\}
\end{equation*}
is a larger minimal potential cover than $\cov$, contradicting the choice of $\cov$.

\medskip\noindent
\textbf{Case II:} \  Symbol $s_1$ is represented at most twice in $\cov$.  

Since $(r_1,c_1,s_1) \in \Urc$, we know that $s_1$ must be represented exactly twice in $\cov$. It follows that $s_1$ cannot occur in both row $r_0$ and column $c_0$. Suppose $s_1$ does not occur in row $r_0$ (the case when $s_1$ does not occur in column $c_0$ can be resolved symmetrically). 


Since $|\cov|\ge2n$, there exists a column $c' \neq c_0$ that is not uniquely represented in $\cov$. Case I implies $s_1$ occurs in column $c_0$ and hence does not occur in column $c'$. Thus,
\begin{equation*}\label{eq:Cmod2}
\big(\cov \setminus \{(r_1,c_1,s_1)\}\big) \cup \big\{(r_0,c_1,s_1), (r_1,c',s_1)\big\}
\end{equation*}
is a larger potential cover than $\cov$, contradicting the choice of $\cov$.



\medskip\noindent
\textbf{Case III:} \  Symbol $s_1$ is represented at least three times in $\cov$.  

In this case,
\begin{equation}\label{eq:Cmod3}
\big(\cov \setminus \{(r_0,c_0,s_0),(r_1,c_1,s_1)\}\big) \cup \big\{(r_0,c_1,s_0), (r_1,c_0,s_0)\big\}
\end{equation}
is another minimal potential cover, with the same cardinality as $\cov$.
The switching \eqref{eq:Cmod3} removes one entry from each of $\Us$ 
and $\Urc$, and replaces them with new entries in $\Uc$ and 
$\Ur$ respectively.

By iteration, we can reach a point where at least one of $\Us$
and $\Urc$ is empty.  A similar process of switchings allows
us to reach a point where one of $\Ur$ and $\Ucs$ is
empty, and also one of $\Uc$ and $\Urs$ is
empty. We continue this process until at least one set from each pair is empty.
Note that while making switch \eqref{eq:Cmod3}, we will increase
the size of two sets in question. However, no matter which switching we perform,
the number of entries in $\Urc \cup \Urs \cup \Ucs$
decreases, so the process terminates.  Call the resulting minimal
potential cover $M$.
 
Note that $M$ satisfies 
$|\Ur| + |\Urc| + |\Urs| < n$, 
since there are only $n$ rows and not all of them are uniquely
represented. Similarly, 
$|\Uc| + |\Urc| + |\Ucs| < n$. 
If $\Us = \emptyset$, then 
\begin{align*}
|\cov| &= |\Ur| + |\Uc|
+ |\Urc| + |\Urs| + |\Ucs| \\
&\le
(|\Ur| + |\Urc| + |\Urs|) + (|\Uc| + |\Urc| + |\Ucs|)
< 2n,
\end{align*}
which contradicts the assumption that $|\cov| \geq 2n$.
Therefore $\Us\neq \emptyset$. 
By similar arguments, $\Ur \neq \emptyset$ and
$\Uc \neq \emptyset$.  By the deductions above, we have
$\Ucs = \Urs = \Urc = \emptyset$, implying
that $|M|=|\Ur|+|\Uc|+|\Us|$.

The entries in $\Us$ cannot share a row with any entry in
$\Ur$, nor share a column with any entry in $\Uc$, so they
lie in an $(n-|\Ur|) \times (n-|\Uc|)$ submatrix, 
implying that $|\Us| \leq
(n-|\Ur|)(n-|\Uc|)$.  Symmetric results hold for
$\Ur$ and $\Uc$, which completes the proof.
\end{proof}

\begin{theo}\label{t:lrgmincov}
Every minimal cover of a Latin square of order $n$ has size at most $\lfloor 3(n+1/2 - \sqrt{n+1/4})\rfloor$. 
\end{theo}

\begin{proof}
Let $x,y,z$ be real numbers from the interval $[0,n]$, and let $\alpha=(x+y+z)/3\in[0,n]$. If $(x,y,z)$ satisfies
\begin{align}
 x+yz &\geq n, \label{eq:ineq1}\\
 y+xz &\geq n \text{ and} \label{eq:ineq2}\\
 z+xy &\geq n, \label{eq:ineq3}
\end{align}
then $(\alpha,\alpha,\alpha)$ also satisfies \eqref{eq:ineq1}--\eqref{eq:ineq3} because
\begin{align*}
\tfrac{1}{3}(x+y+z) + \left( \tfrac{1}{3}(x+y+z) \right)^2 &= \tfrac{1}{9} \left( 3x+3y+3z+x^2+y^2+z^2+2(xy+xz+yz) \right) \\
 & \geq \tfrac{1}{9} \left( 3x+3y+3z+3xy+3xz+3yz \right) \\
 & \geq n,
\end{align*}
where the first inequality holds because $x^2+y^2+z^2 \geq xy+xz+yz$ and the second follows from \eqref{eq:ineq1}--\eqref{eq:ineq3}. Since $\alpha \geq 0$ and $\alpha + \alpha^2 \geq n$, it follows that
\begin{equation}\label{eq:bound-xyz}
 3n - (x+y+z) = 3n - 3\alpha \leq 3\big(n+1/2 - \sqrt{n+1/4}\big).
\end{equation}

Let $M_0$ be an arbitrary minimal cover of a Latin square of order $n$. By \lref{lm:RCSbounds}, there is a minimal potential cover $M$ such that $(n-|\Ur|,n-|\Uc|,n-|\Us|) \in [0,n]^3$ satisfies \eqref{eq:ineq1}--\eqref{eq:ineq3} and $|M| \geq |M_0|$. Thus, by \eqref{eq:bound-xyz}, 
\[
|M_0| \leq |M| = 3n - \big(n-|\Ur|+n-|\Uc|+n-|\Us|\big) \leq 3\big(n+1/2 - \sqrt{n+1/4}\big),
\]
from which the result follows.
\end{proof}

We note that $-1/2 + \sqrt{n+1/4}$ is a positive integer $t$ when $n=t^2+t$.
We next show that the bound in \tref{t:lrgmincov} is achieved for orders $n$ of this form, and therefore, by infinitely many covers.  Moreover, we show that all theoretically possible minimal cover sizes are simultaneously achieved by different covers in a single Latin square of order $t^2+t$.

\begin{lemm} \label{lem:all-covs}
 Let $t\ge2$ and let $L$ be a Latin square of order $n = t^2 + t$ with a transversal $T$ and a minimal cover $\cov$ of size $3t^2$ such that $|\Ur| = |\Uc| = |\Us| = t^2$ and $|\Ur \cap T| = t$. Then $L$ contains a minimal $c$-cover for all $c \in \{t^2+t, \ldots, 3t^2\}$.
\end{lemm}

\begin{proof}
Since $|\Ur| = t^2$, all elements in $\Uc \cup \Us$ must be contained in $t$ rows. Similarly, $\Ur \cup \Us$ must be contained in $t$ columns, and thus, $\Us$ is a $t \times t$ submatrix. We now argue that $\cov\cap T=\Ur\cap T$. Note that $\Us \cap T = \emptyset$ since $\Ur \cup \Us$ is contained in $t$ columns and $|\Ur \cap T| = t$. Similarly, $\Uc \cap T = \emptyset$ since at most $n-|\Us|=t$ distinct symbols occur in $\Ur \cup \Uc$ and $|\Ur \cap T| = t$. 
Permute the rows, columns and symbols of $L$ in such a way that (a)~$T = \{(i,i,i) : 0 \leq i < t^2 + t\}$, (b)~the entries in $\Us$ comprise the bottom-left $t \times t$ submatrix, and (c)~the symbol in the bottom-left entry is $t^2-1$ (this simplifies Case III below).
Thus, $L$ has the following structure:

$$
\begin{array}{c}
\begin{tikzpicture}
 \draw (0,0) -- (0,6) -- (6,6) -- (6,0) -- cycle;
 
 \draw[ultra thick,rounded corners,purple] (0+0.03,0+0.03) rectangle (1.5-0.03,1.5-0.03);
 \draw[ultra thick,rounded corners,green!50!black] (1.5+0.03,0+0.03) rectangle (6-0.03,1.5-0.03);
 \draw[ultra thick,rounded corners,blue] (0+0.03,1.5+0.03) rectangle (1.5-0.03,6-0.03);

 \node at (0.75,0.75) {\color{purple} $\Us$};
 \node at (3.75,0.75) {\color{green!50!black} $\Uc$};
 \node at (0.75,3.75) {\color{blue} $\Ur$};

 \draw[ultra thick,rounded corners,red] (6-0.03,0+0.03) -- (0+0.03,6-0.03);

 \draw [decorate, decoration={brace,amplitude=10pt}]  (1.5-0.05,-0.15) -- (0+0.05,-0.15) node [midway,yshift=-1.5em] {\footnotesize $t$};
 \draw [decorate, decoration={brace,amplitude=10pt}]  (6-0.05,-0.15) -- (1.5+0.05,-0.15) node [midway,yshift=-1.5em] {\footnotesize $t^2$};

 \draw [decorate, decoration={brace,amplitude=10pt}]  (-0.15,0+0.05) -- (-0.15,1.5-0.05) node [midway,xshift=-1.5em] {\footnotesize $t$};
 \draw [decorate, decoration={brace,amplitude=10pt}]  (-0.15,1.5+0.05) -- (-0.15,6-0.05) node [midway,xshift=-1.5em] {\footnotesize $t^2$}; 

\draw [thick,densely dotted,stealth-] (3.1,3.1) to[bend left=5] (6.5,2.5) node[right] (blah) {$T$};

\end{tikzpicture}
\end{array}
$$

Clearly, $T$ itself provides a minimal $(t^2+t)$-cover, and we also know that $L$ has a minimal $(t^2+t+1)$-cover by \tref{th:nplusone}. 
For $c\in \{t^2+t+2, \ldots, 3t^2\}$, we break into 3 cases. In each of these cases, a set of entries from $T$ is added to $\cov$ and then entries that have become redundant are removed. For each entry added that is not in the first $t$ columns nor last $t$ rows, three redundant entries will be removed (one from each of $\Ur$, $\Uc$ and $\Us$). These entries correspond to the set $Y$ below. For each entry added in the last $t$ rows, two redundant entries will be removed (one from each of $\Uc$ and $\Us$). These entries correspond to the set $X$ below.

The $3t$ lines that are not uniquely represented by $\cov$ are (a)~the first $t$ columns, (b)~the last $t$ rows and (c)~the symbols in $\Ur \cap T$. In all cases the modifications that we make leave $\Ur \cap T$ in the resulting cover, so the lines in (a) and (c) will still be represented. The representatives of the last $t$ rows will be addressed in each case. The other checks required to show that the resulting set of entries is a minimal $c$-cover are straightforward and will be omitted. If $Z \subseteq \mathbb{Z}_n$, we define $\Vr(Z) = \{ (i,\unkn,\unkn) \in \Ur : i \in Z \}$, and we define $\Vc$ and $\Vs$ similarly. Whenever we use this notation, the elements in $Z$ will be in one-to-one correspondence with elements of $\Vr$ (similarly for $\Vc$ or $\Vs$).

In each case,
$$\cov(X,Y) =
\big( \cov \cup \left\{ (i,i,i) : i \in X \cup Y \right\}  \big)
\setminus
\big( \Vc(X) \cup \Vs(X) \cup \Vr(Y) \cup \Vc(Y) \cup \Vs(Y) \big).
$$ will be a minimal cover of the appropriate size. Note that in each case, $|\Vc(X) \cup \Vs(X) \cup \Vr(Y) \cup \Vc(Y) \cup \Vs(Y)| = 2|X| + 3|Y|$ and $|\cov(X,Y)| = 3t^2 - |X| - 2|Y|$.

\medskip\noindent
\textbf{Case I:} \
$c \in \{3t^2-t+1, \dots, 3t^2\}$. Define $X = \{t^2,\dots,t^2+(3t^2-c)-1\}$ (with $X = \emptyset$ if $c = 3t^2$) and $Y = \emptyset$.
Note that since $|X|+2|Y| < t$, the elements of $\cov(X,Y)$ in the bottom-left $t \times t$ submatrix cover the last $t$ rows of $L$. Thus, $\cov(X,Y)$ is a minimal $c$-cover.

\medskip\noindent
\textbf{Case II:} \
$c \in \{t^2+t+2, \dots, 3t^2-t\}$ and $c$ is even. Define $X = \{t^2,\dots,t^2+t-1\}$ and $Y = \{t,\dots,t+(3t^2-t-c)/2-1\}$ (with $Y = \emptyset$ if $c = 3t^2 - t$).
Note that the bottom $t$ rows are covered by $\left\{ (i,i,i) : i \in X \right\}$. Thus, $\cov(X,Y)$ is a minimal $c$-cover.

\medskip\noindent
\textbf{Case III:} \
$c \in \{t^2+t+3, \dots, 3t^2-t-1\}$ and $c$ is odd. Define $X = \{t^2,\dots,t^2+t-2\}$ and $Y = \{t,\dots,t+(3t^2-t-c+1)/2-1\}$.
Note that $t^2-1 \not\in Y$, so $(t^2+t-1,0,t^2-1) \in \cov(X,Y)$, and so the bottom $t$ rows are covered by $\left\{ (i,i,i) : i \in X \right\} \cup \{(t^2+t-1,0,t^2-1)\}$. Thus, $\cov(X,Y)$ is a minimal $c$-cover.
\end{proof}

\begin{theo}\label{t:IansConst}
For all $t \geq 2$, there exists a Latin square of order $n=t^2+t$ that contains a minimal $c$-cover for all $c \in \{t^2+t, \ldots, 3t^2\}$.
\end{theo}

\begin{proof}
For each order, we will give an example of a square that satisfies the properties required in \lref{lem:all-covs}. When $t=2$, the following Latin square satisfies the requirements:

$$
\begin{array}{c}
\begin{tikzpicture}
\fill[blue!50] (-3*13pt,2*15pt) -- (-3*13pt,3*15pt) -- (-2*13pt,2*15pt) -- cycle;
\fill[pattern=crosshatch, pattern color=red!50]  (-2*13pt,3*15pt) -- (-3*13pt,3*15pt) -- (-2*13pt,2*15pt) -- cycle;

\fill[blue!50] (-2*13pt,1*15pt) -- (-2*13pt,2*15pt) -- (-1*13pt,1*15pt) -- cycle;
\fill[pattern=crosshatch, pattern color=red!50]  (-1*13pt,2*15pt) -- (-2*13pt,2*15pt) -- (-1*13pt,1*15pt) -- cycle;

\matrix[square matrix]{
5 & 2 & 3 & 0 & 4 & 1 \\
1 & 4 & 0 & 5 & 2 & 3 \\
|[fill=blue!50]| 4 & 0 & |[pattern=crosshatch, pattern color=red!50]| 2 & 3 & 1 & 5 \\
3 & |[fill=blue!50]| 5 & 4 & |[pattern=crosshatch, pattern color=red!50]| 1 & 0 & 2 \\
|[fill=blue!50]| 0 & |[fill=blue!50]| 1 & |[fill=blue!50]| 5 & 2 & |[pattern=crosshatch, pattern color=red!50]| 3 & |[fill=blue!50]| 4 \\
|[fill=blue!50]| 2 & |[fill=blue!50]| 3 & 1 & |[fill=blue!50]| 4 & |[fill=blue!50]| 5 & |[pattern=crosshatch, pattern color=red!50]| 0 \\
};
\end{tikzpicture}
\end{array}
$$
and when $t = 6$, the Latin square given in \fref{fi:t6} in the appendix satisfies the requirements. 

We may now assume that $t \not\in \{2,6\}$, so there exists a pair $(A,B)$ of orthogonal Latin squares of order $t$. Define a $(t^2+t)\times(t^2+t)$ matrix $D$ by filling cell $(\alpha t + r,\beta t + c)$, for $\alpha,\beta \in \{0,\ldots,t\}$ and $r,c \in \{0,\ldots,t-1\}$, with the symbol \[(A_{rc}-(\alpha + \beta +1), B_{rc}) \in \mathbb{Z}_{t+1} \times \mathbb{Z}_{t}.\]
This means, for example, that row $0$ has cell $(0,\beta t + c)$ filled with symbol $(A_{0c}-(\beta +1), B_{0c})$ whenever $0 \leq \beta \leq t$ and $0 \leq c \leq t-1$.
Thus each symbol in $\mathbb{Z}_{t+1} \times \mathbb{Z}_{t}$ occurs exactly once as we iterate over $\beta$ and $c$, so the first row is Latin. 
A similar argument holds for each row and each column, so $D$ is a Latin square. An example of this construction when $t=4$ is given in \fref{f:transvAndMinCover}.

Consider the set of entries in the bottom-left $t \times t$ submatrix of $D$:
\[\mathcal{D}_S = \big\{\big( t^2+r,c, (A_{rc},B_{rc}) \big) \in E(D) : r,c \in \{0,\ldots,t-1\} \big\}.\]
Since $A$ and $B$ are orthogonal, each symbol that occurs in $\mathcal{D}_S$ occurs exactly once. The symbols in $\mathbb{Z}_{t+1} \times \mathbb{Z}_{t}$ that do not occur in $\mathcal{D}_S$ are thus $X=\{(t,0), \ldots, (t,t-1)\}$.  Define
\begin{align*}
\mathcal{D}_R &= \big\{ (r,c,s) \in E(D) : c \in \{0, \ldots, t-1\} \text{ and } s \in X \big\} \text{ and} \\
\mathcal{D}_C &= \big\{ (r,c,s) \in E(D) : r \in \{t^2, \ldots, t^2 + t-1\} \text{ and } s \in X \big\}.
\end{align*}

We next argue that $\cov = \mathcal{D}_R \cup \mathcal{D}_C \cup \mathcal{D}_S$ is a minimal cover of $D$, where $\Ur = \mathcal{D}_R$, $\Uc = \mathcal{D}_C$, and $\Us=\mathcal{D}_S$. Each symbol is covered by $\cov$, as described above. The first $t$ columns are covered by $\mathcal{D}_S$. For any other column $\beta t + c$ (with $\beta \in \{1,\dots,t\}$ and $c \in \{0,\dots,t-1\}$), let $r$ be such that $A_{rc} = \beta-1$. The entry $(t^2+r,\beta t+c,\unkn) \in \mathcal{D}_C$ covers column $\beta t + c$. Since there were $t^2$ such columns to cover and $|\mathcal{D}_C| = t^2$, no entries in $\mathcal{D}_C$ are redundant (all entries in $\mathcal{D}_R$ and $\mathcal{D}_S$ are contained in the first $t$ columns). A similar argument holds for covering the rows. Thus, $\cov$ is a minimal cover of size $3t^2$ with $|\Ur| = |\Uc| = |\Us| = t^2$. Before we can apply Lemma \ref{lem:all-covs}, we must now find a transversal $T$ in $D$ such that $|\Ur \cap T| = t$.

\medskip\noindent
\textbf{Case I:} \
$t$ is even.

We may assume without loss of generality that $A_{rr}=0$ and $B_{rr}=r$ for $r \in \{0,\ldots,t-1\}$. We construct the Latin square $D$ as described above. The symbols on the main diagonal of $D$ are
\[
\{(-2\alpha-1,r) : 0 \leq \alpha  < t+1 \text{ and } 0 \leq r < t \}.
\]
Since $t+1$ is odd, this set is $\mathbb{Z}_{t+1} \times \mathbb{Z}_t$, implying that the main diagonal is a transversal. Note that $\Ur$ intersects the first $t$ entries of the main diagonal. Thus, we may apply \lref{lem:all-covs} to $D$.

\medskip\noindent
\textbf{Case II:} \
$t$ is odd.

We set $A_{rc} = c-r\mod t$ and $B_{rc} = 2c-r \mod t$. Let $D$ be the Latin square from the construction above. Note that
\begin{align*}
\big\{ ( \alpha t + r, \alpha t + r, (-2 \alpha - 1, r)) &: 0 \leq \alpha < (t+1)/2 \text{ and } 0 \leq r < t \big\} \\
\cup \big\{ ( \alpha t + r, \alpha t + (r + 1), (- 2 \alpha, r+2) ) &: (t+1)/2 \leq \alpha < t+1 \text{ and } 0 \leq r < t-1 \big\} \\
\cup  \big\{ ( \alpha t + r, \alpha t, (- 2 \alpha, r+2) ) &: (t+1)/2 \leq \alpha < t+1 \text{ and }r = t-1 \big\}
\end{align*}
is a transversal of $D$ and that $\Ur$ intersects the first $t$ entries of this transversal. Thus, we may apply \lref{lem:all-covs} to $D$.
\end{proof}


\begin{figure}[htp]
\centering
$\begin{array}{ccc}
\begin{array}{c}
\begin{array}{c}A=\end{array}
\begin{array}{c}
\begin{tikzpicture}
\matrix[square matrix]{
0 & 2 & 3 & 1 \\
2 & 0 & 1 & 3 \\
3 & 1 & 0 & 2 \\
1 & 3 & 2 & 0 \\
};
\end{tikzpicture}
\end{array}
\\
\begin{array}{c}B=\end{array}
\begin{array}{c}
\begin{tikzpicture}
\matrix[square matrix]{
0 & 3 & 1 & 2 \\
2 & 1 & 3 & 0 \\
3 & 0 & 2 & 1 \\
1 & 2 & 0 & 3 \\
};
\end{tikzpicture}
\end{array}
\end{array}
&
\begin{array}{c}D=\end{array}
\begin{array}{c}
\begin{tikzpicture}

\fill[blue!50] (-10*13pt,9*15pt) -- (-10*13pt,10*15pt) -- (-9*13pt,9*15pt) -- cycle;
\fill[pattern=crosshatch, pattern color=red!50]  (-9*13pt,10*15pt) -- (-10*13pt,10*15pt) -- (-9*13pt,9*15pt) -- cycle;
\fill[blue!50] (-9*13pt,8*15pt) -- (-9*13pt,9*15pt) -- (-8*13pt,8*15pt) -- cycle;
\fill[pattern=crosshatch, pattern color=red!50]  (-8*13pt,9*15pt) -- (-9*13pt,9*15pt) -- (-8*13pt,8*15pt) -- cycle;
\fill[blue!50] (-8*13pt,7*15pt) -- (-8*13pt,8*15pt) -- (-7*13pt,7*15pt) -- cycle;
\fill[pattern=crosshatch, pattern color=red!50]  (-7*13pt,8*15pt) -- (-8*13pt,8*15pt) -- (-7*13pt,7*15pt) -- cycle;
\fill[blue!50] (-7*13pt,6*15pt) -- (-7*13pt,7*15pt) -- (-6*13pt,6*15pt) -- cycle;
\fill[pattern=crosshatch, pattern color=red!50]  (-6*13pt,7*15pt) -- (-7*13pt,7*15pt) -- (-6*13pt,6*15pt) -- cycle;


\matrix[square matrix]{
16 & 7 & 9 & 2 & 12 & 3 & 5 & 18 & 8 & 19 & 1 & 14 & 4 & 15 & 17 & 10 & 0 & 11 & 13 & 6 \\
6 & 17 & 3 & 8 & 2 & 13 & 19 & 4 & 18 & 9 & 15 & 0 & 14 & 5 & 11 & 16 & 10 & 1 & 7 & 12 \\
11 & 0 & 18 & 5 & 7 & 16 & 14 & 1 & 3 & 12 & 10 & 17 & 19 & 8 & 6 & 13 & 15 & 4 & 2 & 9 \\
1 & 10 & 4 & 19 & 17 & 6 & 0 & 15 & 13 & 2 & 16 & 11 & 9 & 18 & 12 & 7 & 5 & 14 & 8 & 3 \\
12 & 3 & 5 & |[fill=blue!50]| 18 & |[pattern=crosshatch, pattern color=red!50]| 8 & 19 & 1 & 14 & 4 & 15 & 17 & 10 & 0 & 11 & 13 & 6 & 16 & 7 & 9 & 2 \\
2 & 13 & |[fill=blue!50]| 19 & 4 & 18 & |[pattern=crosshatch, pattern color=red!50]| 9 & 15 & 0 & 14 & 5 & 11 & 16 & 10 & 1 & 7 & 12 & 6 & 17 & 3 & 8 \\
7 & |[fill=blue!50]| 16 & 14 & 1 & 3 & 12 & |[pattern=crosshatch, pattern color=red!50]| 10 & 17 & 19 & 8 & 6 & 13 & 15 & 4 & 2 & 9 & 11 & 0 & 18 & 5 \\
|[fill=blue!50]| 17 & 6 & 0 & 15 & 13 & 2 & 16 & |[pattern=crosshatch, pattern color=red!50]| 11 & 9 & 18 & 12 & 7 & 5 & 14 & 8 & 3 & 1 & 10 & 4 & 19 \\
8 & |[fill=blue!50]| 19 & 1 & 14 & 4 & 15 & 17 & 10 & |[pattern=crosshatch, pattern color=red!50]| 0 & 11 & 13 & 6 & 16 & 7 & 9 & 2 & 12 & 3 & 5 & 18 \\
|[fill=blue!50]| 18 & 9 & 15 & 0 & 14 & 5 & 11 & 16 & 10 & |[pattern=crosshatch, pattern color=red!50]| 1 & 7 & 12 & 6 & 17 & 3 & 8 & 2 & 13 & 19 & 4 \\
3 & 12 & 10 & |[fill=blue!50]| 17 & 19 & 8 & 6 & 13 & 15 & 4 & |[pattern=crosshatch, pattern color=red!50]| 2 & 9 & 11 & 0 & 18 & 5 & 7 & 16 & 14 & 1 \\
13 & 2 & |[fill=blue!50]| 16 & 11 & 9 & 18 & 12 & 7 & 5 & 14 & 8 & |[pattern=crosshatch, pattern color=red!50]| 3 & 1 & 10 & 4 & 19 & 17 & 6 & 0 & 15 \\
4 & 15 & |[fill=blue!50]| 17 & 10 & 0 & 11 & 13 & 6 & 16 & 7 & 9 & 2 & |[pattern=crosshatch, pattern color=red!50]| 12 & 3 & 5 & 18 & 8 & 19 & 1 & 14 \\
14 & 5 & 11 & |[fill=blue!50]| 16 & 10 & 1 & 7 & 12 & 6 & 17 & 3 & 8 & 2 & |[pattern=crosshatch, pattern color=red!50]| 13 & 19 & 4 & 18 & 9 & 15 & 0 \\
|[fill=blue!50]| 19 & 8 & 6 & 13 & 15 & 4 & 2 & 9 & 11 & 0 & 18 & 5 & 7 & 16 & |[pattern=crosshatch, pattern color=red!50]| 14 & 1 & 3 & 12 & 10 & 17 \\
9 & |[fill=blue!50]| 18 & 12 & 7 & 5 & 14 & 8 & 3 & 1 & 10 & 4 & 19 & 17 & 6 & 0 & |[pattern=crosshatch, pattern color=red!50]| 15 & 13 & 2 & 16 & 11 \\
|[fill=blue!50]| 0 & |[fill=blue!50]| 11 & |[fill=blue!50]| 13 & |[fill=blue!50]| 6 & |[fill=blue!50]| 16 & 7 & 9 & 2 & 12 & 3 & 5 & |[fill=blue!50]| 18 & 8 & |[fill=blue!50]| 19 & 1 & 14 & |[pattern=crosshatch, pattern color=red!50]| 4 & 15 & |[fill=blue!50]| 17 & 10 \\
|[fill=blue!50]| 10 & |[fill=blue!50]| 1 & |[fill=blue!50]| 7 & |[fill=blue!50]| 12 & 6 & |[fill=blue!50]| 17 & 3 & 8 & 2 & 13 & |[fill=blue!50]| 19 & 4 & |[fill=blue!50]| 18 & 9 & 15 & 0 & 14 & |[pattern=crosshatch, pattern color=red!50]| 5 & 11 & |[fill=blue!50]| 16 \\
|[fill=blue!50]| 15 & |[fill=blue!50]| 4 & |[fill=blue!50]| 2 & |[fill=blue!50]| 9 & 11 & 0 & |[fill=blue!50]| 18 & 5 & 7 & |[fill=blue!50]| 16 & 14 & 1 & 3 & 12 & 10 & |[fill=blue!50]| 17 & |[fill=blue!50]| 19 & 8 & |[pattern=crosshatch, pattern color=red!50]| 6 & 13 \\
|[fill=blue!50]| 5 & |[fill=blue!50]| 14 & |[fill=blue!50]| 8 & |[fill=blue!50]| 3 & 1 & 10 & 4 & |[fill=blue!50]| 19 & |[fill=blue!50]| 17 & 6 & 0 & 15 & 13 & 2 & |[fill=blue!50]| 16 & 11 & 9 & |[fill=blue!50]| 18 & 12 & |[pattern=crosshatch, pattern color=red!50]| 7 \\
};

\draw[ultra thick] (-10*13pt,-10*15pt) rectangle (10*13pt,10*15pt);

\draw[ultra thick] (-6*13pt,-10*15pt) to (-6*13pt,10*15pt);
\draw[ultra thick] (-2*13pt,-10*15pt) to (-2*13pt,10*15pt);
\draw[ultra thick] (2*13pt,-10*15pt) to (2*13pt,10*15pt);
\draw[ultra thick] (6*13pt,-10*15pt) to (6*13pt,10*15pt);

\draw[ultra thick] (-10*13pt,-6*15pt) to (10*13pt,-6*15pt);
\draw[ultra thick] (-10*13pt,-2*15pt) to (10*13pt,-2*15pt);
\draw[ultra thick] (-10*13pt,2*15pt) to (10*13pt,2*15pt);
\draw[ultra thick] (-10*13pt,6*15pt) to (10*13pt,6*15pt);



\end{tikzpicture}
\end{array}
\end{array}$
\caption{\label{f:transvAndMinCover}Example of the construction in the proof of \tref{t:IansConst} after the symbols are relabeled to belong to $\mathbb{Z}_{20}$.  Here we have $t=4$, and we highlight a $3t^2$-cover.  We also highlight the main diagonal, which is a transversal.}
\end{figure}
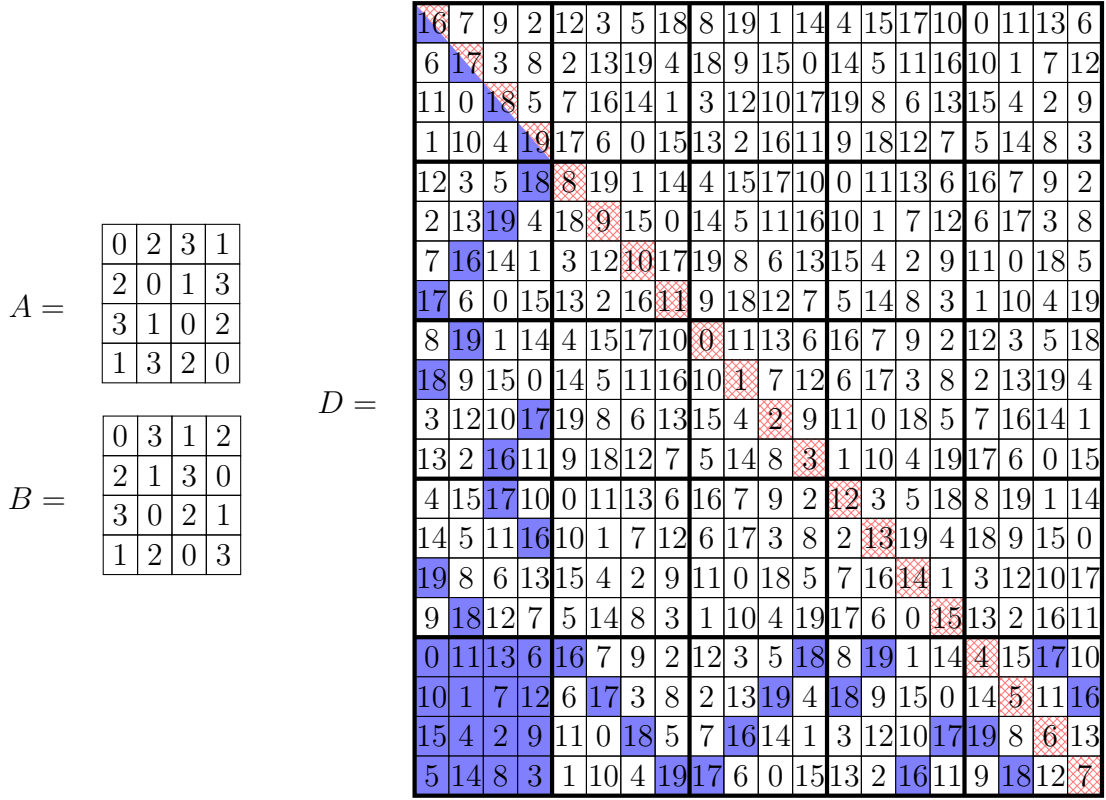

Our next goal is to show that all Latin squares have a minimal cover that is asymptotically equal to the bound in \tref{t:lrgmincov}. To do so, we introduce the notion of a partial minimal cover.  If $L$ is a partial Latin square and $\mathcal{P} \subseteq E(L)$ such that, for some $\mathbf{e} \in \mathcal{P}$, both $\mathcal{P}$ and $\mathcal{P} \setminus \{\mathbf{e}\}$ represent the same lines, then we call $\mathbf{e}$ \emph{redundant}.  An entry $(r,c,s) \in \mathcal{P}$ is redundant if and only if there exists three other entries of the form $(r,\unkn,\unkn)$, $(\unkn,c,\unkn)$ and $(\unkn,\unkn,s)$ in $\mathcal{P}$.  We define a \emph{partial minimal cover} as any $\mathcal{P} \subseteq E(L)$ that has no redundant entries.  We can iteratively delete redundant entries from any $\mathcal{P} \subseteq E(L)$ to obtain a partial minimal cover of size no more than $|\mathcal{P}|$ in which the same lines are represented.



It is important to note that not every partial minimal cover can be extended to a minimal cover, and \fref{f:non-extendable} gives two examples of partial minimal covers that cannot be extended to a minimal cover (nor even a larger partial minimal cover).

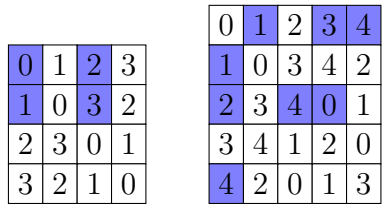
\begin{figure}[htp]
\begin{center}
\begin{tikzpicture}
\matrix[square matrix]{
|[fill=blue!50]| 0 & |[fill=white]| 1 & |[fill=blue!50]| 2 & |[fill=white]| 3 \\
|[fill=blue!50]| 1 & |[fill=white]| 0 & |[fill=blue!50]| 3 & |[fill=white]| 2 \\
|[fill=white]| 2 & |[fill=white]| 3 & |[fill=white]| 0 & |[fill=white]| 1 \\
|[fill=white]| 3 & |[fill=white]| 2 & |[fill=white]| 1 & |[fill=white]| 0 \\
};
\end{tikzpicture}
\quad
\begin{tikzpicture}
\matrix[square matrix]{
|[fill=white]| 0 & |[fill=blue!50]| 1 & |[fill=white]| 2 & |[fill=blue!50]| 3 & |[fill=blue!50]| 4 \\
|[fill=blue!50]| 1 & |[fill=white]| 0 & |[fill=white]| 3 & |[fill=white]| 4 & |[fill=white]| 2 \\
|[fill=blue!50]| 2 & |[fill=white]| 3 & |[fill=blue!50]| 4 & |[fill=blue!50]| 0 & |[fill=white]| 1 \\
|[fill=white]| 3 & |[fill=white]| 4 & |[fill=white]| 1 & |[fill=white]| 2 & |[fill=white]| 0 \\
|[fill=blue!50]| 4 & |[fill=white]| 2 & |[fill=white]| 0 & |[fill=white]| 1 & |[fill=white]| 3 \\
};
\end{tikzpicture}
\caption{\label{f:non-extendable}Two Latin squares with partial minimal covers that are not subsets of any minimal cover.}
\end{center}
\end{figure}

Even though a partial minimal cover does not necessarily extend to a minimal cover, we can find a minimal cover that is at least as large as any partial minimal cover.

\begin{lemm}\label{l:partial-to-full}
Let $L$ be a Latin square of order $n$ and $\mathcal{P}$ be a partial minimal cover of $L$. Then $L$ contains a minimal cover of size at least $|\mathcal{P}|$.
\end{lemm}

\begin{proof}
If $|\mathcal{P}| \leq 2n-1$, then the cover described in the proof of \lref{l:trivcov} satisfies the constraints, so assume $|\mathcal{P}| \geq 2n$.  If $\mathcal{P}$ is a minimal cover, then the statement is trivial so suppose there is some line, say row $r$, that is not covered by $\mathcal{P}$.


Since $|\mathcal{P}| \geq 2n$, there exists a column $c$ that is represented at least twice in $\mathcal{P}$. Define $\mathcal{P}' = \mathcal{P} \cup \{(r,c,s)\}$ where $s=L_{rc}$.  If $\mathcal{P}'$ is not a partial minimal cover, then there must be some entry in row $r$, in column $c$ or with symbol $s$ that is redundant. Since $(r,c,s)$ is the only entry in row $r$, it is not redundant. Since there are at least three entries in column $c$ in $\mathcal{P}'$, no entry in column $c$ is redundant in $\mathcal{P}'$ (otherwise we contradict the minimality of $\mathcal{P}$). However, if $s$ is represented exactly once in $\mathcal{P}$, by $(r_0,c_0,s)$ say, then that entry is redundant in $\mathcal{P}'$ if and only if there are other entries covering row $r_0$ and column $c_0$. In this case, we define $\mathcal{P}''=\mathcal{P}' \setminus \{(r_0,c_0,s)\}$, otherwise, we define $\mathcal{P}''=\mathcal{P}'$. Note that in either case, $\mathcal{P}''$ is a partial minimal cover that covers strictly more lines than $\mathcal{P}$ and is at least as big as $\mathcal{P}$.

We repeat the above process until all lines are covered.
\end{proof}


Next, we need a technical lemma.

\begin{lemm}\label{l:bipgr}
Fix $\eps>0$. Let $G$ be a bipartite graph with bipartition $A\cup B$ and maximum degree at most $n^{1/2+\eps}$. Suppose that $G$ has $n^{3/2+\eps}-O(n^{1+2\eps})$ edges and that $|A|=n-O(n^{1/2+\eps})$ and $|B|=n-O(n^{1/2+\eps})$.  Then we can find a set $U\subset A$ of vertices such that  $|U|=O(n^{1/2+\eps})$ and there are at most $O(n^{1/2+2\eps})$ vertices in $B$ that do not have a neighbour in $U$.
\end{lemm}

\begin{proof}
Let $B'\subseteq B$ be the set of vertices in $B$ of degree at least $\frac12 n^{1/2+\eps}$. Counting edges, we have that
\[
n^{3/2+\eps}-O(n^{1+2\eps}) \le n^{1/2+\eps}|B'|+\tfrac12 n^{1/2+\eps}\big(n-O(n^{1/2+\eps})-|B'|\big),
\]
which implies that $|B'|\ge n-O(n^{1/2+2\eps})$. Consider choosing a set $U \subseteq A$ of size $|U|=\lceil n^{1/2+\eps}\rceil$ uniformly at random. For any $v\in B'$ the probability that $v$ has no neighbour in $U$ is
\[
\frac{\binom{|A|-\mathrm{deg}(v)}{|U|}}{\binom{|A|}{|U|}} \leq 
\frac{\binom{|A|-\frac12 n^{1/2+\eps}}{|U|}}{\binom{|A|}{|U|}}
=\prod_{0\le i<|U|}\frac{|A|-\frac12 n^{1/2+\eps}-i}{|A|-i}
\le \bigg(\frac{|A|-\frac12 n^{1/2+\eps}}{|A|}\bigg)^{|U|}
=O\big(\exp(-\tfrac12 n^{2\eps})\big)
\]
using the identity $1-1/x \leq e^{-1/x}$ when $x \neq 0$. So the expected number of vertices in $B'$ with no neighbour in $U$ is $O\big(n\exp(-\frac12 n^{2\eps})\big)=o(1)$. It follows that for large $n$ there is some choice of $U$ whose neighbourhood includes $B'$, and we are done.
\end{proof}


\begin{theo}\label{th:LSmincov}
Fix $\eps>0$.  Every Latin square of order $n$ has a minimal cover of size $3n-O(n^{1/2+\eps})$.
\end{theo}

\begin{proof}
If $\eps \geq 1/2$, then the theorem follows from \lref{l:trivcov}, so assume that $\eps < 1/2$. Suppose $L$ is a Latin square of order $n$ and let  $\psi=\lfloor n^{1/2+\eps}\rfloor$.  We gradually build a large partial minimal cover $\cov$ for $L$.

Define $B_1$ to be the $\psi$-regular bipartite graph with vertices $\{c_0,\dots,c_{n-1}\}\cup\{s_0,\dots,s_{n-1}\}$ with an edge $c_is_j$ if and only if $L_{ki}=j$ for some $k \in\{0,\dots,\psi-1\}$.  Applying \lref{l:bipgr} to $B_1$, we find a set $U_1\subseteq\{c_0,\dots,c_{n-1}\}$ with $|U_1|=O(n^{1/2+\eps})$ such that $n-O(n^{1/2+2\eps})$ vertices in $\{s_0,\dots,s_{n-1}\}$ have a neighbour in $U_1$.  In other words, the submatrix $S$ formed by the rows indexed $\{0,\dots,\psi-1\}$ and the columns indexed $\{i : c_i \in U_1\}$ contains a set of $n-O(n^{1/2+2\eps})$ entries with distinct symbols, and we (provisionally) initialise $\cov$ to be this set of entries.  By removing at most $\psi$ entries from $\cov$ if necessary, we identify a set $\Psi$ of $\psi$ symbols that are not yet represented in $\cov$.

Next we form a bipartite graph $B_2$. The vertices of $B_2$ correspond to the rows and columns of $L$ that do not intersect $S$. We place an edge from row vertex $r$ to column vertex $c$ if and only if $L_{rc}\in\Psi$.  Since $S$ has $O(n^{1/2+\eps})$ rows and $O(n^{1/2+\eps})$ columns, $B_2$ has $n\psi-O(n^{1/2+\eps}\psi)=n^{3/2+\eps}-O(n^{1+2\eps})$ edges and maximum degree at most $\psi$.  Hence we can apply \lref{l:bipgr} twice to find a set $U_2$ of rows and a set $U_3$ of columns with desired properties that we now describe. First, they do not intersect $S$.  Second, they are small enough that $|U_2|=O(n^{1/2+\eps})$ and $|U_3|=O(n^{1/2+\eps})$.  We (provisionally) include in $\cov$ any entry containing a symbol in $\Psi$ in the rows in $U_2$ and/or the columns of $U_3$. \lref{l:bipgr} implies that these entries cover a set $U_4$ of $n-O(n^{1/2+2\eps})$ rows and a set $U_5$ of $n-O(n^{1/2+2\eps})$ columns.



At this point, $\cov$ may not be a partial minimal cover, so we iteratively remove redundant entries from $\cov$. Afterwards, the following three sets, each comprising of $n-O(n^{1/2+2\eps})$ lines, are covered and no entry in $\cov$ can cover more than one of the following lines:
\begin{itemize}
 \item The rows in $U_4$ that are not in $U_2$.
 \item The columns in $U_5$ that are not in $U_3$.
 \item The symbols other than those in $\Psi$.
\end{itemize}
Thus $\cov$ is a partial minimal cover of size at least $3n-O(n^{1/2+2\eps})$.  By \lref{l:partial-to-full}, there is a minimal cover of $L$ of size $3n-O(n^{1/2+2\eps})$. We replace $\eps$ by $\eps/2$ to complete the proof.
\end{proof}

Next, we report on some computations of sizes of minimal covers for small Latin squares.

The Cayley tables of the groups $\Z_3$ and $\Z_2\times\Z_2$ have transversals and $(n+2)$-covers, but do not have any minimal $(n+1)$-covers. Thus the spectrum of sizes of minimal covers is not continuous in these two cases. However, these two Latin squares may be just small anomalies, since we found no other Latin squares of order up to $8$ with a gap in their spectrum.

For orders $n\le5$, the minimal cover constructed in \lref{l:trivcov} meets the bound in \tref{t:lrgmincov} and hence has maximum possible size. For each order in the range $6\le n\le 9$, we found a Latin square that has no minimal cover meeting the bound in \tref{t:lrgmincov}. Our computations were exhaustive for $6\le n\le 8$, where there is a gap of only 1 between the size of the cover in \lref{l:trivcov} and the bound in \tref{t:lrgmincov}. For $n=6$, there are 6 species that meet the bound and 6 that do not; neither group table meets the bound. For $n=7$ there are 145 species that meet the bound. The 2 species that do not meet the bound contain the group $\Z_7$ and the Steiner quasigroup. For $n=8$ there are 283654 species that meet the bound. The 3 species that do not meet the bound contain the dihedral group, the elementary abelian group, and 
the Latin square obtained by turning an intercalate in the elementary abelian group (that is, by replacing a $2\times 2$ Latin subsquare with the other possible subsquare on the same two symbols). Note that the autotopism group of the elementary abelian group acts transitively on the intercalates, so it does not matter which intercalate gets turned.

We could not do exhaustive computations for all Latin squares of order $9$, but we confirmed that $\Z_3\times\Z_3$ meets the bound in \tref{t:lrgmincov}, whilst $\Z_9$ does not. The largest minimal cover in $\Z_9$ has size 18, which is one more than the size of the example in \lref{l:trivcov} but one less than the bound in \tref{t:lrgmincov}.

\bigskip

In \sref{s:duality}, we showed a kind of duality between minimal covers and maximal partial transversals. However, we next reveal a distinction between the behaviours of these objects.  We begin with the following theorem, which gives the values of $k$ and $n$ for which there exists a Latin square of order $n \geq 5$ with a maximal partial transversals of deficit $d=n-k$.

\begin{theo}\label{th:parttrans_subsq}
For all integers $n \geq 5$ and $k \geq 1$ satisfying $n \geq 2k$,  there exists a Latin square $L=[L_{ij}]$ of order $n$ where $L_{ii}=i$ for all $i \in \{0,\ldots,n-k-1\}$ and the intersection of the $k$ rows and columns indexed by $\{n-k,\ldots,n-1\}$ is a subsquare (i.e., a submatrix that is a Latin square) on the symbols $\{0,\ldots,k-1\}$.  Consequently, $L$ has a maximal partial transversal of length $n-k$.
\end{theo}

\begin{proof}
A Latin square $M=[M_{ij}]$ of order $m$ is idempotent if $M_{ii}=i$ for all $i \in \{0,\dots,m-1\}$. 
Any Latin square with a transversal can be made idempotent by applying an isotopism.

The $n=2k$ case of the theorem is immediate, by simply taking a direct product of an idempotent Latin square of order $k$ with a Latin square of order $2$. So we may assume that $k \leq n-k-1$. Also, note that $n-k \geq \lceil n/2 \rceil \geq 3$.

\begin{figure}[htb]
\centering

\begin{tikzpicture}
\matrix[square matrix]{
|[fill=white]| 0 & |[fill=white]| 2 & |[fill=blue!50]| 9 & |[fill=blue!50]| 8 & |[fill=white]| 3 & |[fill=white]| 1 \\
|[fill=white]| 3 & |[fill=white]| 1 & |[fill=white]| 0 & |[fill=blue!50]| 7 & |[fill=white]| 2 & |[fill=blue!50]| 4 \\
|[fill=blue!50]| 9 & |[fill=white]| 0 & |[fill=white]| 2 & |[fill=white]| 1 & |[fill=blue!50]| 6 & |[fill=white]| 3 \\
|[fill=blue!50]| 8 & |[fill=blue!50]| 7 & |[fill=white]| 1 & |[fill=white]| 3 & |[fill=white]| 0 & |[fill=white]| 2 \\
|[fill=white]| 1 & |[fill=white]| 3 & |[fill=blue!50]| 5 & |[fill=white]| 2 & |[fill=blue!50]| 4 & |[fill=white]| 0 \\
|[fill=white]| 2 & |[fill=blue!50]| 6 & |[fill=white]| 3 & |[fill=white]| 0 & |[fill=white]| 1 & |[fill=blue!50]| 5 \\
};
\draw [decorate,decoration={brace,amplitude=10pt}] (-3*13pt,3.2*15pt) -- (3*13pt,3.2*15pt) node [black,midway,yshift=0.8cm] {$\substack{\text{\small $M'$ when}\\[0.2em] \text{\small $n-k=6$ and $k=4$}}$};
\end{tikzpicture}
\qquad
\begin{tikzpicture}
\matrix[square matrix]{
|[fill=white]| 0 & |[fill=white]| 4 & |[fill=white]| 1 & |[fill=blue!50]| 6 & |[fill=white]| 2 & |[fill=white]| 3 \\
|[fill=white]| 3 & |[fill=white]| 1 & |[fill=white]| 4 & |[fill=white]| 2 & |[fill=blue!50]| 7 & |[fill=white]| 0 \\
|[fill=white]| 4 & |[fill=blue!50]| 8 & |[fill=white]| 2 & |[fill=white]| 0 & |[fill=white]| 3 & |[fill=white]| 1 \\
|[fill=white]| 1 & |[fill=white]| 2 & |[fill=blue!50]| 9 & |[fill=white]| 3 & |[fill=white]| 0 & |[fill=white]| 4 \\
|[fill=blue!50]| 10 & |[fill=white]| 0 & |[fill=white]| 3 & |[fill=white]| 1 & |[fill=white]| 4 & |[fill=white]| 2 \\
|[fill=white]| 2 & |[fill=white]| 3 & |[fill=white]| 0 & |[fill=white]| 4 & |[fill=white]| 1 & |[fill=blue!50]| 5 \\
};
\draw [decorate,decoration={brace,amplitude=10pt}] (-3*13pt,3.2*15pt) -- (3*13pt,3.2*15pt) node [black,midway,yshift=0.8cm] {$\substack{\text{\small $M'$ when}\\[0.2em] \text{\small $n-k=6$ and $k=5$}}$};
\end{tikzpicture}
\caption{\label{fi:nk6}Matrices used in the $n-k=6$ case of the proof of \tref{th:parttrans_subsq}.}
\end{figure}

If $n-k=6$ and $k\in\{4,5\}$, we define $M'$ as given in \fref{fi:nk6}. In all other relevant cases, we can find a Latin square of order $n-k$ with $k+1$ disjoint transversals \cite{transurv}. Applying an isotopism, we get an idempotent Latin square $M=[M_{ij}]$ of order $n-k$. It has $k$ disjoint transversals, denoted $d_\sigma$ for $\sigma \in \{n-k,\ldots,n-1\}$, which do not intersect the main diagonal.  We replace the symbols in $\{k,\ldots,n-k-1\}$ in each $d_\sigma$ by the symbol $\sigma$, and call the result $M'$.  We give an example of this construction in \fref{fi:edit}. 

Thus, $M'$ is idempotent and contains $n-k$ copies of each symbol in $\{0,\ldots,k-1\}$ and $n-2k$ copies of each symbol in $\{k,\ldots,n-1\}$.  Ryser's Theorem \cite{Rys51} implies that $M'$ embeds in a Latin square $L$ of order $n$; this is illustrated for the example in \fref{fi:edit}.  Moreover, since $M'$ contains each symbol in $\{0,\ldots,k-1\}$ exactly $n-k$ times, the intersection of the $k$ rows and columns indexed by $\{n-k,\ldots,n-1\}$ in $L$ must be a subsquare on the symbols $\{0,\ldots,k-1\}$.
\end{proof}

\begin{figure}[htb]
\centering
$\begin{array}{ccccc}
\begin{array}{c}
\begin{tikzpicture}
\matrix[square matrix]{
0 & |[fill=blue!50]| 4 & |[fill=red!50]| 3 & 2 & 1 \\
2 & 1 & |[fill=blue!50]| 0 & |[fill=red!50]| 4 & 3 \\
4 & 3 & 2 & |[fill=blue!50]| 1 & |[fill=red!50]| 0 \\
|[fill=red!50]| 1 & 0 & 4 & 3 & |[fill=blue!50]| 2 \\
|[fill=blue!50]| 3 & |[fill=red!50]| 2 & 1 & 0 & 4 \\
};
\end{tikzpicture}
\end{array}
&
\begin{array}{c}
\xrightarrow{\text{edit}}
\end{array}
&
\begin{array}{c}
\begin{tikzpicture}
\matrix[square matrix]{
0 & |[fill=blue!50]| 5 & |[fill=red!50]| 6 & 2 & 1 \\
2 & 1 & |[pattern=crosshatch, pattern color=blue!50]| 0 & |[fill=red!50]| 6 & 3 \\
4 & 3 & 2 & |[pattern=crosshatch, pattern color=blue!50]| 1 & |[pattern=crosshatch, pattern color=red!50]| 0 \\
|[pattern=crosshatch, pattern color=red!50]| 1 & 0 & 4 & 3 & |[fill=blue!50]| 5 \\
|[fill=blue!50]| 5 & |[fill=red!50]| 6 & 1 & 0 & 4 \\
};
\end{tikzpicture}
\end{array}
&
\begin{array}{c}
\xrightarrow{\text{embed in Latin square}}
\end{array}
&
\begin{array}{c}
\begin{tikzpicture}
\matrix[square matrix]{
0 & 5 & 6 & 2 & 1 & 3 & 4 \\
2 & 1 & 0 & 6 & 3 & 4 & 5 \\
4 & 3 & 2 & 1 & 0 & 5 & 6 \\
1 & 0 & 4 & 3 & 5 & 6 & 2 \\
5 & 6 & 1 & 0 & 4 & 2 & 3 \\
3 & 2 & 5 & 4 & 2 & 0 & 1 \\
6 & 4 & 3 & 5 & 6 & 1 & 0 \\
};

\draw[ultra thick,rounded corners] (-3.5*13pt,-1.5*15pt) rectangle (1.5*13pt,3.5*15pt);
\end{tikzpicture}
\end{array}
\end{array}$
\caption{\label{fi:edit}Example of the construction in the proof of \tref{th:parttrans_subsq} when $n=7$ and $k=2$.}
\end{figure}

Any partial transversal of length less than $\lceil n/2\rceil$ can be extended.  Thus, a consequence of \tref{th:parttrans_subsq} is that among all Latin squares of order $n \geq 5$, the shortest maximal partial transversal has length $\lceil n/2\rceil$.  \tref{th:LSmincov} shows that the upper bound on minimal covers described in \tref{t:lrgmincov} is achieved asymptotically for all Latin squares of order $n$.  However, as we establish in the following theorem, most Latin squares do not come close to achieving a maximal partial transversal of length $\lceil n/2\rceil$.  While minimum covers directly relate to maximum partial transversals (see Theorems~\ref{th:PTtoCover} and~\ref{th:CovertoPT}), maximum minimal covers seem not to have a direct relationship with minimum maximal partial transversals.

\begin{theo}\label{t:RLSminmaxpt}
Fix $\eps>0$. With probability approaching $1$ as $n\rightarrow\infty$, a Latin square of order $n$ chosen uniformly at random has no maximal partial transversal of deficit exceeding $n^{2/3+\eps}$.
\end{theo}

\begin{proof}
Let $L$ be a random Latin square of order $n$. Suppose that $L$ has a maximal partial transversal $T$ of deficit $d$. Let $S$ be the $d\times d$ submatrix of $L$ induced by the rows and columns that are not represented in $T$. By the maximality of $T$, we know that $S$ contains none of the $d$ symbols that are not represented in $T$. However, if this is the case and $d=n^{2/3+\eps}$, then \cite[Thm~2]{KS17} would imply that $n^{1+3\eps} = d^3/n = O(n^{1+3\eps/2}\log n)$, which is a contradiction, so no such submatrix $S$ exists in $L$.
\end{proof}


\section{Concluding remarks}\label{s:conclude}

We have introduced covers of Latin squares with the aim of using them to better understand partial transversals, focusing primarily on topics relating to extremal sizes.

We found that some properties of covers have analogous properties for partial transversals, while others do not.  For example, the maximum size of partial transversals is closely related to the minimum size of covers.  However, the minimum size of a maximal partial transversal is $\lceil n/2\rceil$, which most Latin squares do not come close to achieving (see \tref{t:RLSminmaxpt}). In contrast, the maximum size of a minimal cover is $3n-O(n^{1/2})$, which is asymptotically achieved by all Latin squares (see \tref{th:LSmincov}).

There are $(n+1)$-covers that contain no partial transversals of deficit $0$ or $1$.  The error on the upper bound on the number of partial transversals in \tref{th:nplus1} grows with the number of such $(n+1)$-covers.  Also, while Brualdi's Conjecture implies the existence of $(n+1)$-covers in all Latin squares of order $n$, we have not established the converse.  Instead, a weaker form of the converse is true: if every Latin square of order $n \geq 2$ has an $(n+1)$-cover, then every Latin square of order $n \geq 2$ has a partial transversal of deficit $2$.

Relating the enumeration of partial transversals with small deficit ($d \in \{1,2\}$) to the enumeration of $(n+1)$-covers is also difficult because the number of embeddings of a maximal partial transversal of deficit $d$ within an $(n+1)$-cover depends on the structure of the Latin square.

There are switches that can be performed among $(n+1)$-covers, such as
\[
\begin{array}{ccc}
\begin{array}{c}
\begin{tikzpicture}
\matrix[square matrix]{
|[fill=white]| 3 & |[fill=blue!50]| 1 & |[fill=white]| 5 & |[fill=white]| 2 & |[fill=white]| 0 & |[fill=white]| 4 \\
|[fill=white]| 1 & |[fill=white]| 0 & |[fill=blue!50]| 3 & |[fill=white]| 5 & |[fill=white]| 4 & |[fill=white]| 2 \\
|[fill=white]| 4 & |[fill=white]| 2 & |[fill=white]| 1 & |[fill=white]| 3 & |[fill=white]| 5 & |[fill=blue!50]| 0 \\
|[fill=white]| 2 & |[fill=white]| 5 & |[fill=white]| 0 & |[fill=blue!50]| 4 & |[fill=white]| 1 & |[fill=white]| 3 \\
|[fill=blue!50]| 0 & |[fill=white]| 4 & |[fill=white]| 2 & |[fill=white]| 1 & |[fill=white]| 3 & |[fill=blue!50]| 5 \\
|[fill=white]| 5 & |[fill=white]| 3 & |[fill=white]| 4 & |[fill=white]| 0 & |[fill=blue!50]| 2 & |[fill=white]| 1 \\
};
\end{tikzpicture}
\end{array}
&
\begin{array}{c}
\longleftrightarrow
\end{array}
&
\begin{array}{c}
\begin{tikzpicture}
\matrix[square matrix]{
|[fill=white]| 3 & |[fill=blue!50]| 1 & |[fill=white]| 5 & |[fill=white]| 2 & |[fill=white]| 0 & |[fill=white]| 4 \\
|[fill=white]| 1 & |[fill=white]| 0 & |[fill=blue!50]| 3 & |[fill=white]| 5 & |[fill=white]| 4 & |[fill=white]| 2 \\
|[fill=blue!50]| 4 & |[fill=white]| 2 & |[fill=white]| 1 & |[fill=white]| 3 & |[fill=white]| 5 & |[fill=white]| 0 \\
|[fill=white]| 2 & |[fill=white]| 5 & |[fill=white]| 0 & |[fill=blue!50]| 4 & |[fill=white]| 1 & |[fill=white]| 3 \\
|[fill=blue!50]| 0 & |[fill=white]| 4 & |[fill=white]| 2 & |[fill=white]| 1 & |[fill=white]| 3 & |[fill=blue!50]| 5 \\
|[fill=white]| 5 & |[fill=white]| 3 & |[fill=white]| 4 & |[fill=white]| 0 & |[fill=blue!50]| 2 & |[fill=white]| 1 \\
};
\end{tikzpicture}
\end{array}
\end{array}
\]
which converts an $(n+1)$-cover inducing $G_5$ into an $(n+1)$-cover inducing $G_3$.  However, we did not succeed in making switchings work for converting $(n+1)$-covers inducing $G_1$ into the other structures, which would yield a partial transversal of deficit $1$.  It is possible that more complicated switching patterns might succeed in changing the graph structure in $(n+1)$-covers inducing $G_1$, but it is also possible that identifying such switchings would not be possible without, say, proving Brualdi's Conjecture.

In the case of minimal covers of maximum size, the results in \sref{s:maxmincov} make significant progress, finding an explicit upper bound that is achieved infinitely often, and that is achieved asymptotically by all Latin squares.

In the proof of \tref{th:LSmincov}, we find an  $O(n^{1/2+\eps}) \times O(n^{1/2+\eps})$ submatrix $S$ containing all but $O(n^{1/2+\eps})$ symbols.  This raises the question as to whether stronger results in this direction hold. Does every $n^2 \times n^2$ Latin square contain an $n \times n$ submatrix that contains every symbol?  The $4 \times 4$ Latin squares each have $2 \times 2$ submatrices containing all four symbols, but the $9 \times 9$ Latin square
\begin{center}
\begin{tikzpicture}
\matrix[square matrix]{
  0 & 1 & 2 & 3 & 4 & 5 & 6 & 7 & 8 \\
  1 & 0 & 3 & 2 & 5 & 6 & 7 & 8 & 4 \\
  2 & 3 & 1 & 0 & 7 & 8 & 4 & 5 & 6 \\
  3 & 2 & 0 & 1 & 6 & 7 & 8 & 4 & 5 \\
  4 & 5 & 8 & 7 & 1 & 2 & 3 & 6 & 0 \\
  5 & 6 & 4 & 8 & 0 & 1 & 2 & 3 & 7 \\
  6 & 7 & 5 & 4 & 8 & 0 & 1 & 2 & 3 \\
  7 & 8 & 6 & 5 & 3 & 4 & 0 & 1 & 2 \\
  8 & 4 & 7 & 6 & 2 & 3 & 5 & 0 & 1 \\
};
\end{tikzpicture}
\end{center}
found by White \cite{White}, has the property that no $3 \times 3$ submatrix contains all nine symbols. It would be of some interest to find more precise results for general Latin squares as to how small a submatrix contains every symbol, and/or how many distinct symbols we can be sure to find in at least one submatrix of given dimensions.

There are multiple directions in which the study of covers could be extended; we describe some below.

Some of the results here could be extended to Latin rectangles or even special kinds of partial Latin rectangles such as plexes \cite{transurv}.  It would also be interesting to extend the investigation to Latin hypercubes, sets of mutually orthogonal Latin squares, or to MDS codes more generally.

The Cayley tables of groups are of particular interest, since transversals in them are equivalent to orthomorphisms, and problems such as enumeration of orthomorphisms (particularly for cyclic groups) have been studied \cite{MMW2006}.  Moreover, cyclic group tables have a lot of structure (see, e.g. \lref{l:grpG1-5}) that may permit a more successful study of switchings than in general Latin squares.


Each of the five structurally distinct $(n+1)$-covers can be embedded in a Latin square of order $5$, as shown in \fref{fig:subgraphs}, so by replacing the $5\times 5$ subsquares in the $k=5$ case of \tref{th:parttrans_subsq}, we find that every potential $(n+1)$-cover embeds in a Latin square of order $n$, for all $n \geq 10$.  In fact, the same is easily found to be true for orders in $\{5,\ldots,9\}$ (by searching random Latin squares of these orders).  It would be interesting to resolve the general case of this embedding problem, i.e., for which orders $n$ does every potential $(n+a)$-cover complete to a Latin square?  A famous problem along these lines is Evan's Conjecture \cite{Evans}, which has since been proved \cite{AndersonHilton,Haggkvist,Smetaniuk}, which states that a partial Latin square of order $n$ with at most $n-1$ entries can be completed.

Balasubramanian \cite{Balasubramanian1990} showed that Latin squares of even order have an even number of transversals. Exhaustive computations for orders $n\le8$ suggest the following:

\begin{conj}
Let $L$ be a Latin square of even order $n$, with $t$ transversals and $\qmin$ minimal $(n+1)$-covers. Then $t\equiv2\qmin\mod4$.
\end{conj}
Another curious observation is that the number of $(n+1)$-covers in every Latin square of order $7$ is divisible by $3$.

Finally, we mention that the data in \tabref{tab:avg-q_i} shows approximate consistency in the number of $(n+1)$-covers that Latin squares of order $n$ have. If this is a pattern, it might be worth investigating as a means to prove a weakened form of Brualdi's conjecture (via \tref{th:CovertoPT}).


\appendix
\newpage
\section{Sporadic example}
\begin{figure}[!htp]
\centering
\includegraphics[width=\textwidth]{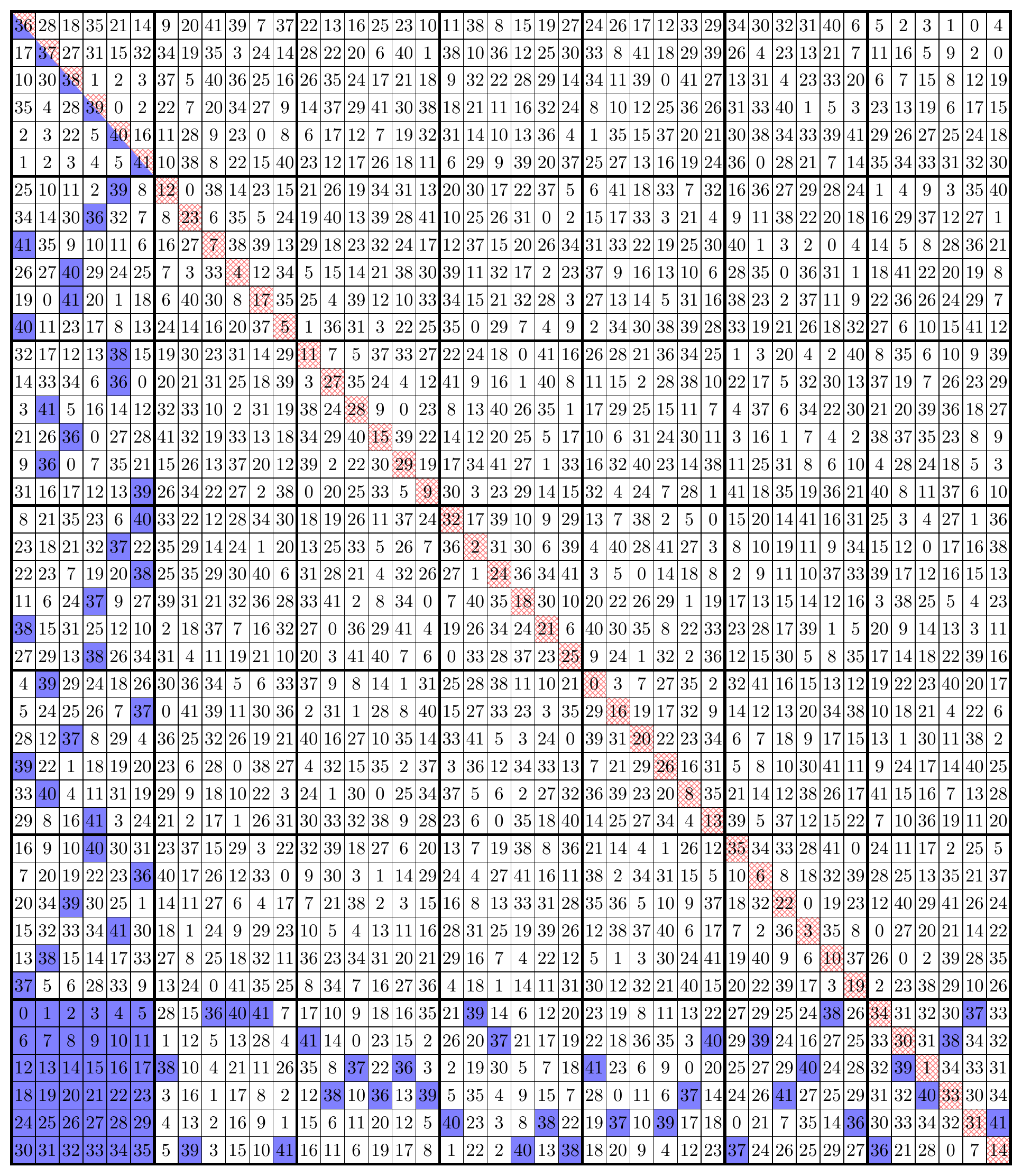}
\caption{\label{fi:t6}A Latin square of order $42$ and a minimal $108$-cover generated by a semi-random computer search.  The main diagonal is a transversal.}
\end{figure}

\end{document}